\documentclass[11pt]{amsart}
\usepackage[margin=1in]{geometry}
\usepackage{graphicx}
\usepackage{soul}
\usepackage{amsthm, amsmath, amssymb, bm, bbm,euscript}
\usepackage{thmtools}
\usepackage{tikz}
\usetikzlibrary{positioning}
\usetikzlibrary{arrows.meta}
\usepackage{mathrsfs}
\usepackage{stmaryrd}

\usepackage{nicematrix}
\usepackage{tikz-cd}

\usepackage[utf8]{inputenc}
\usepackage[T1]{fontenc}
\usepackage[textsize=scriptsize,backgroundcolor=orange!5]{todonotes}

\usepackage{hyperref}
\usepackage{url}

\usepackage{mathtools}

\usepackage[noabbrev,capitalize]{cleveref}
\crefname{equation}{}{}
\usepackage{xcolor}

\usepackage{xcolor,graphics, graphicx}
\usepackage{floatrow}
\usepackage{verbatim}
\newfloatcommand{capbtabbox}{table}[][\FBwidth]

\numberwithin{equation}{section}

\newtheorem{theorem}{Theorem}[section]
\newtheorem{proposition}[theorem]{Proposition}
\newtheorem{lemma}[theorem]{Lemma}
\newtheorem{lemma/defn}[theorem]{Lemma/Definition}

\newtheorem{corollary}[theorem]{Corollary}
\newtheorem{conjecture}[theorem]{Conjecture}
\newtheorem*{question*}{Question}

\theoremstyle{definition}
\newtheorem{definition}[theorem]{Definition}

\newtheorem{notation}[theorem]{Notation}

\theoremstyle{remark}
\newtheorem{remark}[theorem]{Remark}

\newcommand{\mc}{\mathcal}

\renewcommand{\subset}{\subseteq}
\renewcommand{\supset}{\supseteq}

\newcommand{\R}{\mathbb{R}}
\newcommand{\N}{\mathbb{N}}
\newcommand{\Z}{\mathbb{Z}}
\newcommand{\C}{\mathbb{C}}
\newcommand{\F}{\mathbb{F}}

\newcommand{\Q}{\mathbb{Q}}

\DeclareMathOperator{\ind}{ind}

\DeclareMathOperator{\es}{\varnothing}

\newcommand{\kbar}{\overline{k}}

\newcommand{\Hom}{\mathrm{Hom}}

\renewcommand{\subset}{\subseteq}

\newcommand{\actson}{\curvearrowright}
\newcommand{\Aut}{\mathrm{Aut}}

\newcommand{\Qbar}{\overline{\mathbb{Q}}}
\newcommand{\Gal}{\mathrm{Gal}}

\renewcommand{\mod}{\ {\rm mod} \ }

\newcommand{\disc}{{\rm disc}}

\newcommand{\Heis}{{\rm Heis}}

\newcommand{\Br}{{\rm Br}}

\usepackage[OT2,T1]{fontenc}
\DeclareSymbolFont{cyrletters}{OT2}{wncyr}{m}{n}
\DeclareMathSymbol{\Sha}{\mathalpha}{cyrletters}{"58}

\title{A refined Malle conjecture for Heisenberg groups}

\author{Jack B. Miller}
\address{Department of Mathematics \\
    Harvard University \\
    Science Center Room 325 \\
    1 Oxford Street \\
    Cambridge, MA 02138 USA}
\email{jmiller@math.harvard.edu}

\author{Tim Santens}
\address{
	University of Cambridge \\ 
	DPMMS \\
	Centre for Mathematical Sciences\\
	Wilberforce Road \\
	Cambridge \\
	CB3 0WB \\ UK}
\email{ts996@cam.ac.uk}
	
\date{\today}

\allowdisplaybreaks

\begin{document}

\begin{abstract}
Based on a conjecture of Loughran and the second author, we give an explicit prediction for the leading constant in Malle's conjecture for Galois $\mathcal{H}$-extensions of $\mathbb{Q}$ ordered by discriminant, where $\mathcal{H}$ is the $3\times 3$ Heisenberg group over $\mathbb{F}_4$.
The predicted leading constant is not a single Euler product, but rather a sum of two distinct Euler products.
Our methods also give an efficient algorithm for computing the conjectural Loughran–Santens leading constant for many $2$-groups of nilpotency class $2$.
\end{abstract}

\maketitle

\section{Introduction}

Let $G \leq S_n$ be a transitive permutation group. Malle's conjecture predicts an asymptotic formula
\[
N(G,X) \sim c(G)\, X^{a(G)} (\log X)^{b(G)-1}
\]
for the number $N(G,X)$ of degree $n$ number fields with Galois group $G$ and absolute discriminant bounded by $X$, where the exponents $a(G)$ and $b(G)$ are determined by the permutation-theoretic structure of $G$ \cite{Malle2002,Malle2004}.
Since Kl\"uners' counterexample to the original logarithmic exponent \cite{Klueners2005}, a great deal of work has gone into formulating corrected versions of Malle's conjecture and proving them in new cases; see, for example, the history of number field counting results given in \cite[\S1.4]{ALOWWinductivecounting}.
Even when the exponents are understood, however, the \textit{leading constant} for $G$, denoted $c(G)$, remains far more mysterious.

In many known cases of Malle's conjecture, the leading constant for $G$ is given by an Euler product
\[
C_{\rm MB}(G) = \alpha(G) \prod_{v} C_{{\rm MB},v}(G),
\]
where $\alpha(G) \in \Q_{> 0}$ is a group-cohomological constant, and for each place $v$ of $\Q$ the local factor $C_{{\rm MB},v}(G)$ is interpreted as a normalized count of $G$-\'etale algebras over $\Q_v$, often referred to as the (normalized) \textit{$v$-adic local mass} of $G$.
When the identity $c(G) = C_{\rm MB}(G)$ holds for a permutation group $G$, it is very likely a numerical shadow of the \textit{Malle--Bhargava principle for $G$} \cite{Malle2002,Malle2004,BhargavaMassFormulaS_n}.
We give a rigorous definition of the Malle--Bhargava principle that we call \textit{independence of local events} (\cref{def:local events and independence and almost-everywhere independence}).
Examples of permutation groups $G$ known to satisfy independence of local events include:
\begin{itemize}
    \item Elementary abelian groups \cite{WrightAbelianExtensions,WoodAbelianLocalProbabilities};
    \item $S_n$ for $2 \le n \le 5$, where $S_n$ is in its standard representation \cite{DavenportHeilbronn,BhargavaQuarticRings,BhargavaQuinticRings,BSWgeometryofnumbers1};
    \item $D_4$ in its octic representation \cite{ShankarVarma2025}.
\end{itemize}

There is a well-known conjectural obstruction to independence of local events for $G$ known as \textit{concentration}.
We define the \textit{index function} $\ind :G \to \Z_{\ge 0}$ via
\[
\ind(g) := n - \#\{\text{cycles of }g \in S_n\}.
\]
Malle \cite{Malle2002} conjectured that the exponent of $X$ appearing in the asymptotic for $N(G,X)$ is given by
\[
a(G) := 
(\min\{\ind(g) : g\in G\backslash\{1\}\})^{-1},
\]
and we refer to the set $\{g \in G : \ind(g) = a(G)^{-1}\}$ as the \textit{minimal index elements} of the permutation group $G$.
If the minimal index elements generate $G$ as a group, we say that $G$ is \textit{non-concentrated} \cite{ALOWWinductivecounting}, otherwise $G$ is \textit{concentrated}.
An equivalent terminology for $G$ being non-concentrated is that the $G$-discriminant is a \textit{balanced height function} \cite[\S9.1]{LoughranSantens2024}.

All currently known cases of Malle's conjecture for $G$-number fields ordered by discriminant fall into one of the following two categories:
\begin{enumerate}
    \item $G$ is concentrated;
    \item $G$ satisfies independence of local events.
\end{enumerate}
In this paper we find a permutation group $G$ such that, conditional on a conjecture of Loughran and the second author \cite{LoughranSantens2024}, we show that:
\begin{enumerate}
    \item[(3)] $G$ is non-concentrated and fails to satisfy independence of local events.
\end{enumerate}

\begin{conjecture}
    \label{conj:explicit Heis_4 conjecture}
    Let $\Heis_4$ be the order $64$ Heisenberg group over $\F_4$ viewed in its regular permutation representation $\Heis_4 \le S_{64}$.
    We conjecture that $\Heis_4$ fails to satisfy independence of local events, and that
    \[
    \#\left\{K/\Q \textnormal{ Galois} :
    \begin{matrix}
        |\disc(K/\Q)| \le X, \\
        \Gal(K/\Q) \simeq \Heis_4
    \end{matrix}
    \right\} 
    \sim \frac{1}{2} \, (C_0 + C_\beta) 
    \, X^{1/32} \, (\log X)^8,
    \]
    where
    \[
    C_0 = \alpha_0 
    \prod_{p > 2} \left(
        1 + \frac{9}{p} + \frac{9}{p^{3/2}}
    \right)
    \left(
        1 - \frac{1}{p}
    \right)^9, \qquad 
    C_{\beta} =
    \alpha_{\beta}
    \prod_{p > 2} \left(
        1 + \frac{9}{p} + \frac{3}{p^{3/2}}
    \right)
    \left(
        1 - \frac{1}{p}
    \right)^9,
    \]
    \[
    \alpha_0 = \frac{313+54\sqrt{2}+54\sqrt[4]{2}+9\sqrt{2}\sqrt[4]{2}}{2^{71} \cdot 3^4 \cdot 5},
    \qquad 
     \alpha_{\beta} = \frac{247 + 30\sqrt[4]{2} + 36\sqrt{2} - 3\sqrt{2}\sqrt[4]{2}}{2^{71} \cdot 3^3 \cdot 5}.
    \]
\end{conjecture}

We prove (\cref{thm:derivation from Loughran Santens conjecture}) that \cref{conj:explicit Heis_4 conjecture} is a special case of \cite[Conj.\ 9.3, Conj.\ 9.10]{LoughranSantens2024} for the base field $\Q$ and permutation group $G = \Heis_4 \le S_{64}$.
In particular, we note that the stack $B\Heis_4$ has a transcendental partially unramified Brauer--Manin obstruction relevant to counting by Galois discriminant (\cref{lem:partially_ramified_Brauer_group}), and the surjective locus of $BG$ avoids thin sets expected to contain a positive proportion of $\Q$-rational points when ordering by discriminant (\cref{prop:thin sets do not contribute}).

\subsubsection*{Interpretation of the leading constant}
\label{subsubsec:interpretation of leading constant}

If the stack $BG$ did not have a Brauer--Manin obstruction, then the leading constant $c(G) = \frac{1}{2}(C_0 + C_{\beta})$ would equal $C_0$.
Instead, the constant $C_{\beta} \approx 0.633\, C_0$ comes from a transcendental Brauer--Manin obstruction.
By interpreting the constant $c(G)$ as in \cite[Conj.\ 9.10]{LoughranSantens2024}, we can directly observe a failure of independence of local events.
Indeed, the Brauer--Manin obstruction causes the conjectured covariance between odd primes $p_1 \ne p_2$ ramifying in a random Galois $\Heis_4$-number field ordered by discriminant to be equal to
\[
\frac{36 C_0C_{\beta}}{(C_0+C_\beta)^2} \frac{1}{p_1^{3/2} p_2^{3/2}}
\left(1 + O\left(\frac{1}{p_1} + \frac{1}{p_2}\right)\right).
\]
The fact that this covariance is nonzero for infinitely many pairs of distinct odd primes implies that $\Heis_4$ fails to satisfy independence of local events.
In fact, $\Heis_4$ fails to satisfy the weaker condition of \textit{almost-everywhere independence of local events} (\cref{def:local events and independence and almost-everywhere independence}).
\newline

We note that \cref{conj:explicit Heis_4 conjecture} is not completely out of reach.
Indeed, it seems likely that the nilpotent counting methods of Koymans--Pagano \cite{KoymansPagano2023} combined with the character sum method can be used to prove this conjecture.
For example, Fouvry--Koymans \cite{FouvryKoymans} used these methods to count $\Heis_3$-extensions ordered by discriminant and Hansen--Zanoli \cite{Hansen2025Counting} recently used them to prove a multi-height version of Malle's conjecture for $D_4 \simeq \Heis_2$.
It would be interesting to confirm this suspicion.
Recent work of Alberts--Bucur \cite{AlbertsBucurMDS} may also be applicable to proving Malle's conjecture for $\Heis_4 \le S_{64}$ (with inexplicit leading constant).

Previous work considered the closely related question of whether failure of almost-everywhere independence of local events could occur for \textit{some} balanced height function.
\begin{itemize}
    \item In correspondence with Alberts, Wood \cite[\S7.6]{AlbertsRandomGroupMBmodel} constructed a balanced height for $\Z/4\Z$-number fields that Tavernier \cite[\S4.5]{TavernierAbelianRestrictedRam} proved fails to satisfy almost-everywhere independence of local events when ordering by the height.
    \item Loughran and the second author \cite[Conj.\ 1.2]{LoughranSantens2024} constructed a balanced height for $A_4$-number fields that conjecturally fails to satisfy almost-everywhere independence of local events when ordering by the height.
\end{itemize}
The first example is due to an algebraic Brauer--Manin obstruction, and the second example is due to a transcendental Brauer--Manin obstruction.
The balanced height functions are chosen so that a non-trivial Brauer class of the relevant classifying stack is unramified for every minimal weight conjugacy class of the height, but ramifies at some other conjugacy class, resulting in a leading constant that is a finite sum of distinct Euler products (cf.\ \cite[Rem.\ 8.25]{LoughranSantens2024}).

While the above balanced height functions exhibit a failure of almost-everywhere independence of local events, they do not correspond to Artin conductors \cite[Ex.\ 4.19]{LoughranSantensSurvey}, let alone a discriminant height function.

In order to arrive at \cref{conj:explicit Heis_4 conjecture}, we carry out a systematic study of the partially unramified Brauer groups of $BG$ that are responsible for influencing the asymptotics of discriminant orderings of $G$-number fields.

\subsection{Heisenberg groups in characteristic $2$}

The explicit form of \cref{conj:explicit Heis_4 conjecture} comes from the study of Heisenberg groups in their regular representations.

\begin{definition}
    \label{def:Heisenberg group}
    Let $q$ be a prime power.
    We define the \textit{Heisenberg group} over $\F_q$ to be the matrix group
    \[
    \Heis_q := \left\{
    \begin{pmatrix}
        1&x&z\\
        &1&y\\
        &&1
    \end{pmatrix} : x,y,z \in \F_q
    \right\}.
    \]
\end{definition}

Previously, Fouvry and Koymans \cite{FouvryKoymans} developed theory surrounding Malle's conjecture for the permutation groups $\Heis_{\ell} \le S_{\ell^2}$ where $\ell$ is an odd prime and $\Heis_{\ell}$ acts on the affine plane $\F_{\ell}.(1,0,1) \oplus \F_{\ell}.(0,1,1) \subset \F_{\ell}^3$.
In order to find interesting Brauer--Manin obstructions over $\Q$, we instead focus on $\Heis_{2^n} \le S_{8^n}$ where $\Heis_{2^n}$ is in its regular representation.

We summarize our results for Heisenberg groups that play a role in \cite[Conj.\ 9.3]{LoughranSantens2024} as follows.

\begin{theorem}
    \label{thm:intro theorem for Heis_2^n results}
    Let $n \in \Z_{\ge 1}$, and $C \subset \Heis_{2^n}$ be the subset of order $2$ elements.
    \begin{enumerate}
        \item The Brauer group $\Br_{\rm un}^{e} B\Heis_{2^n}$ is trivial.
        \item The Brauer group $\Br_C^{e} B\Heis_{2^n}$ is isomorphic to the group of $\F_2$-bilinear forms on $\F_{2^n}$ modulo the pullback of $\F_2$-linear functionals $\F_{2^n} \to \F_2$ along the multiplication form $\F_{2^n} \times \F_{2^n} \to \F_{2^n}$.
        In particular, $|\Br_C^{e} B\Heis_{2^n}| = 2^{n(n-1)}$.
        \item Let $v$ be a place of $\Q$.
        When $v \ne 2$, there are simple formulae for the $v$-adic mass (and Brauer transforms) of $B\Heis_{2^n}$ (see \cref{cor:odd prime local Brauer transforms,cor:archimedean local Brauer transforms}).
        When $v=2$, there is an efficient group-theoretic algorithm for computing the $2$-adic mass (and Brauer transforms) of $B\Heis_{2^n}$ (see \cref{cor:2-adic mass in terms of admissible triples}).
    \end{enumerate}
\end{theorem}

By \cref{prop:thin sets do not contribute} and \cite[Conj.\ 9.3]{LoughranSantens2024}, \cref{thm:intro theorem for Heis_2^n results}(2) predicts that the leading constant in Malle's conjecture for Galois $\Heis_{2^n}\,$-number fields ordered by discriminant decomposes as a finite sum of $2^{n(n-1)}$ Euler products.
In the case of $n=1$, this prediction for the leading constant agrees with the result of Shankar--Varma \cite[Thm.\ 1.1]{ShankarVarma2025} on Galois $D_4$-fields ordered by discriminant.
Indeed, in the case of $D_4 \simeq \Heis_2$ there is only one Euler product, and independence of local events holds.
We note that this is an exception rather than the rule:

\begin{corollary}
Let $n \in \Z_{\ge 1}$, and assume that \cite[Conj.\ 9.10]{LoughranSantens2024} holds for the base field $\Q$, constant group scheme $\Heis_{2^n}$, and height given by the Galois discriminant.
Then the permutation group $\Heis_{2^n}\le S_{8^n}$ satisfies independence of local events if and only if $n=1$.
\end{corollary}

\subsection{Two-step nilpotent groups}

Many of our methods generalize from Heisenberg groups to two-step nilpotent groups.
\cref{thm:intro theorem for Heis_2^n results} naturally generalizes to analogous statements for finite groups $G$ that are given as a central extension
\[
0 \to W \to G \to V \to 0
\]
where $V,W$ are finite-dimensional $\F_2$-vector spaces.
Such finite groups $G$ correspond to all $2$-groups of nilpotency class $\le 2$ and exponent $\le 4$.

The study of the unramified Brauer group $\Br_{\rm un} BG$ for such groups $G$ goes back (in a different language) to work of Saltman \cite{Saltman1984Noether}.
He constructed a nilpotent group of this shape with a non-trivial unramified Brauer group over $\C$ and used this group to provide a counterexample to Noether's problem. Saltman's analysis was extended by Bogomolov \cite{Bogomolov1988Brauer}, who in particular computed the unramified Brauer group over $\C$ for all groups of the above shape.

There are two issues with extending the analysis of Bogomolov to our case.
The first minor issue is that the relevant Brauer group is not the unramified Brauer group, but rather a \textit{partially unramified} Brauer group.
The second, more substantial, issue is that we need the Brauer group over $\Q$ rather than $\C$.
To deal with these issues, we use \cite[Lem.~6.10]{LoughranSantens2024} to describe the partially unramified Brauer groups in terms of marked central extensions.
 
The most technical part of the paper is the computation of the $2$-adic local Brauer transforms.
We do this as follows.

Let $\mc{L}$ denote the set of isomorphism classes of finite $2$-groups of nilpotency class $\le 2$ and exponent $\le 4$.
We show that $\mc{L}$ is a \textit{level} in the sense of \cite{SawinWoodMomentMethod} and explicitly compute the pro-$\mc{L}$ completion of the absolute Galois group of $\Q_2$ in addition to its ramification filtration.

\begin{theorem}
    \label{thm:intro thm presentation and ramification filtration at 2}
    The group $\Gamma_{\Q_2}^{\mc{L}}$ is finite of order $256$, with presentation
    \[
     \Gamma_{\Q_2}^{\mc{L}} \simeq
    \langle \sigma_{-1}, \sigma_{5}, \sigma_{2} \mid \sigma_{-1}^2 = [\sigma_5,\sigma_2]\rangle^{\mc{L}},
    \]
    and the lower numbering ramification filtration is explicitly given in \cref{prop:ramification filtration}.
\end{theorem}

\subsection{Outline of the paper}

In \cref{sec:derivation of LS conjecture}, we prove that \cref{conj:explicit Heis_4 conjecture} follows from \cite[Conj.\ 9.3]{LoughranSantens2024}.
We also prove the triviality of algebraic partially unramified Brauer groups in many cases relevant to Malle's conjecture (\cref{lem:vanishing of algebraic Brauer groups in many cases}).

In \cref{sec:two step nilpotent groups}, we describe the partially unramified Brauer group of a non-concentrated regular permutation group $G$ of even order in terms of central extensions (\cref{lem:Brauer_group_discriminant}).
When $G$ is a two-step nilpotent $2$-group, we classify the partially unramified Brauer elements in terms of quadratic forms (\cref{subsec:two-step nilpotent groups}).

In \cref{sec:the case of Heisenberg groups}, we specialize our regular permutation group $G$ to be $\Heis_{2^n}$, and prove most of \cref{thm:intro theorem for Heis_2^n results} with the exception of the $\infty$- and $2$-adic Brauer transforms of $B\Heis_{2^n}$.

In \cref{sec:first section on 2-adic mass}, we explicitly compute the $\infty$- and $2$-adic Brauer transforms of $B\Heis_4$ by assuming the local ramification filtration for $\Gamma_{\Q_2}^{\mc{L}}$ (\cref{thm:intro thm presentation and ramification filtration at 2}).

In \cref{sec:second section on 2-adic mass}, we prove \cref{thm:intro thm presentation and ramification filtration at 2}.

\subsection{Notation}

For every perfect field $k$ appearing in this paper, we let $\Gamma_k := \Gal(\kbar/k)$ where $\kbar/k$ denotes a fixed algebraic closure.
For every place $v$ of $\Q$, we fix an inclusion $\Gamma_{\Q_v} \subset \Gamma_{\Q}$.

For every finite separable extension of fields $L/K$, we let $\Gal(L/K)$ denote the permutation group whose underlying abstract group is $\Gal(\widetilde{L}/K)$, where $\widetilde{L}/K$ is the normal closure of $L/K$, and the permutation representation is given by the natural action of $\Gal(\widetilde{L}/K)$ on the set $\Hom_{K}(L, \overline{K})$.

\textbf{Brauer transforms:} Let $G \le S_n$ be a transitive permutation group and $v$ a place of $\Q$.
We define the \textit{$v$-adic mass} of $G$ to be the real number
\[
\tau_{\disc,v}(BG) := \frac{1}{|G|} \sum_{\varphi_v : \Gamma_{\Q_v} \to G} \frac{1}{|\disc(\varphi_v)|^{a(G)}},
\]
where $|\disc(\varphi_v)| \in \Z_{\ge 1}$ is the absolute norm of the discriminant of the $G$-structured \'etale algebra over $\Q_v$ corresponding to $\varphi_v$ in the case $v$ is non-archimedean.
In the case $v$ is archimedean, we set $|\disc(\varphi_{\infty})| := 1$.

We now let $\beta \in \Br^e BG$ be a Brauer class \cite[Lem.\ 6.2]{LoughranSantensSurvey}.
We define the \textit{$v$-adic Brauer transform} of $G$ with respect to $\beta$ to be the complex number
\[
\widehat{\tau}_{\disc,v}(BG;\beta) := \frac{1}{|G|} \sum_{\varphi_v : \Gamma_{\Q_v} \to G} \frac{e^{2 i \pi {\rm inv}_v\, \beta(\varphi_v)}}{|\disc(\varphi_v)|^{a(G)}},
\]
where $\beta(\varphi_v) \in \Br \, \Q_v$ is the \textit{evaluation map} (cf.\ \cite[Def.\ 6.6]{LoughranSantensSurvey}) and ${\rm inv}_v: \Br \, \Q_v \to \Q/\Z$ is the local invariant map from local class field theory.

We observe that $\widehat{\tau}_{\disc,v}(BG;0) = \tau_{\disc,v}(BG)$ is the $v$-adic mass of $G$.

\subsection*{Acknowledgements}

The collaboration that led to this paper was started during the thematic program on universal statistics at the Centre de Recherches Mathématiques (CRM) at the University of Montréal.
The authors would like to thank the CRM for its hospitality and financial support.

We would like to thank Daniel Loughran and Melanie Matchett Wood for useful comments.

The first author was partially supported by NSF DMS-2140043, the James Mills Peirce Fellowship at Harvard University, and a Churchill Scholarship.
The second author was supported by the Herschel Smith Fund.

The authors acknowledge the use of AI tools, including GPT-5.5 and Codex by OpenAI, for autonomously writing initial Magma code that was used to search for small permutation groups admitting a Brauer--Manin obstruction, and suggesting the statements of \cref{prop:ramification filtration} and \cref{cor:numerical 2-adic mass of Heis_4}.
The authors are responsible for all statements, proofs, and exposition in this paper.

\section{Derivation of \cref{conj:explicit Heis_4 conjecture}}
\label{sec:derivation of LS conjecture}

In this section, we give a definition of (almost-everywhere) independence of local events and derive \cref{conj:explicit Heis_4 conjecture} using the framework of \cite{LoughranSantens2024}.
We discuss why the transcendental Brauer--Manin obstruction appearing in \cref{conj:explicit Heis_4 conjecture} is the simplest possible kind of obstruction by proving that there are no algebraic Brauer--Manin obstructions when ordering by discriminant.
We also give a brief description of how the methods of our paper naturally give an algorithm for computing the leading constant in Malle's conjecture for various $2$-groups ordered by Galois discriminant.

\begin{definition}
    \label{def:local events and independence and almost-everywhere independence}
    Let $n\in \Z_{\ge 1}$ and $G \le S_n$ be a transitive permutation group.
    
    $\bullet$ If $v$ is a place of $\Q$, a \textit{$v$-local event} for $G$ is a subset $\Sigma_v \subset \Hom(\Gamma_{\Q_v},G)$ that is $G$-conjugation invariant.
    A \textit{local event} is a tuple $\Sigma = (\Sigma_v)_{v\in S}$ where $S$ is a finite set of places of $\Q$ and $\Sigma_v$ is a $v$-local event.
    Given a local event $\Sigma$ and $X \in \R_{\ge 1}$, we define
    \[
    N(G,X;\Sigma) := \#\left\{\Q \subset K \subset \Qbar :
    \begin{matrix}
        \exists\, \varphi : \Gamma_{\Q} \twoheadrightarrow G, \ \ker(\varphi) = \Gal(\Qbar/K) \\
        \forall v \in S, \ \varphi|_{\Gamma_{\Q_v}} \in \Sigma_v \\
        |\disc(K/\Q)| \le X
    \end{matrix}
    \right\}.
    \]
    We let $N(G,X) := N(G,X;\es)$.

    $\bullet$ We say that $G$ satisfies \textit{independence of local events} if $N(G,X) > 0$ for $X$ sufficiently large, and for every local event $\Sigma = (\Sigma_{v})_{v\in S}$ one has that the following limits exist and are equal:
    \[
    \lim_{X \to \infty}
    \frac{N(G,X;\Sigma)}{N(G,X)} = \lim_{X\to \infty}
    \prod_{v \in S} \frac{N(G,X;\Sigma_{v})}{N(G,X)}.
    \]
    
    $\bullet$ We say that $G$ satisfies \textit{almost-everywhere independence of local events} if $N(G,X) > 0$ for $X$ sufficiently large, and there exists a finite set $T$ of places of $\Q$ such that for every local event $\Sigma = (\Sigma_{v})_{v\in S}$ where $S \cap T = \es$ one has that the following limits exist and are equal:
    \[
    \lim_{X \to \infty}
    \frac{N(G,X;\Sigma)}{N(G,X)} = \lim_{X\to \infty}
    \prod_{v \in S} \frac{N(G,X;\Sigma_{v})}{N(G,X)}.
    \]
\end{definition}

\begin{remark} \hfill
	\begin{enumerate}
		\item Analogous definitions can be made for all big height functions \cite[\S4.4]{LoughranSantensSurvey}.
		\item Wood defines a notion of fair counting functions for abelian groups $G$ and shows that almost-everywhere independence of local events holds in this case; see \cite[Cor.~2.2]{WoodAbelianLocalProbabilities}.
	\end{enumerate}
\end{remark}

\begin{theorem}
    \label{thm:derivation from Loughran Santens conjecture}
    The following statements hold.
    \begin{enumerate}
        \item Assume \cite[Conj.\ 9.3]{LoughranSantens2024} for the base field $\Q$ and permutation group $\Heis_4 \le S_{64}$.
        Then the explicit asymptotic counting statement in \cref{conj:explicit Heis_4 conjecture} holds.
        \item Assume \cite[Conj.\ 9.10]{LoughranSantens2024} holds for the base field $\Q$, constant group scheme $\Heis_{4}$, and balanced height given by the Galois discriminant. Then the permutation group $\Heis_4 \le S_{64}$ fails to satisfy almost-everywhere independence of local events.
    \end{enumerate}
\end{theorem}

\begin{proof}[Proof of \textnormal{(1)}]
    By \cref{prop:thin sets do not contribute}, thin sets do not contribute to the conjectural asymptotic number of Galois $\Heis_4$-number fields ordered by discriminant.

    We record Malle's $a$- and $b$-constants.
    Since $\Heis_4 \le S_{64}$ is in its regular representation and $2$ is the minimal prime dividing its order, it is clear that $a(\Heis_4)^{-1} = |\Heis_4|(1-\frac{1}{2}) = 32$.
    The minimal index conjugacy classes correspond to the order $2$ conjugacy classes, and by \cref{lem:conjugacy_classes} there are $9$ of them.
    Since order $2$ elements are fixed under invertible powering, this implies that $b(\Heis_4) = 9$.
    We let $a = 1/32$ and $b = 9$.

    \cite[Conj.\ 9.3]{LoughranSantens2024} asks us to compute a groupoid cardinality conversion factor of
    \[
    \frac{|Z_{S_{64}}(\Heis_{4})||\Heis_4|}{|N_{S_{64}}(\Heis_4)|}.
    \]
    Since $\Heis_4 \le S_{64}$ is a regular representation, this ratio simplifies to $|\Heis_4|\cdot|\Aut(\Heis_4)|^{-1}$.
    Since $\Heis_4$ corresponds to \cite[\href{https://www.lmfdb.org/Groups/Abstract/64.242}{\texttt{64.242}}]{LMFDB}, we use \cite{LMFDB} to observe that $|\Aut(\Heis_4)| = 2^{10} \cdot 3^{2}$.

    Let $C \subset \Heis_4$ denote the subset of order $2$ elements.
    \cite[Conj.\ 9.3]{LoughranSantens2024} asks us to compute the constant
    \[
    \frac{a^{b-1} \cdot |\Br_C^{e} B\Heis_4|}{|\Hom(\Heis_4,\Q^{\times})| \cdot (b-1)!}.
    \]
    We use \cref{thm:intro theorem for Heis_2^n results}(2) to deduce that $|\Br_C^{e} B\Heis_4| = 2^{2}$.
    We also have that $|\Hom(\Heis_4,\Q^{\times})| = |\Heis_4^{\rm ab}| = 2^{4}$.

    Lastly, \cite[Conj.\ 9.3]{LoughranSantens2024} asks us to compute for all $\beta \in \Br_C^{e} B\Heis_4$ the global Brauer transform \cite[Lem.\ 7.6]{LoughranSantensSurvey}
    \[
    \prod_{p > 2} \left(1 - \frac{1}{p}\right)^{b}
    \widehat{\tau}_{\disc,p}(B\Heis_{4};\beta) \cdot 
    \widehat{\tau}_{\disc,\infty}(B\Heis_{4};\beta) \cdot 
    \left(1 - \frac{1}{2}\right)^{b}
    \widehat{\tau}_{\disc,2}(B\Heis_{4};\beta).
    \]
    \cref{thm:intro theorem for Heis_2^n results}(3) implies that the above constant is equal to
    \[
    \widehat{\tau}(\beta) := 
    \begin{cases}
    \prod_{p > 2} \left(1 - \frac{1}{p}\right)^{9}\left(1 + \frac{9}{p} + \frac{9}{p^{3/2}}\right) \cdot 
    \frac{7}{16} \cdot 
    \left(1 - \frac{1}{2}\right)^9
    \frac{313+54\sqrt{2}+54\sqrt[4]{2}+9\sqrt{2}\sqrt[4]{2}}{16} & \beta \textnormal{ is trivial},
    \\ 
    \prod_{p > 2} \left(1 - \frac{1}{p}\right)^{9}\left(1 + \frac{9}{p} + \frac{3}{p^{3/2}}\right) \cdot 
    \frac{7}{16} \cdot 
    \left(1 - \frac{1}{2}\right)^9
    \frac{247 + 30\sqrt[4]{2} + 36\sqrt{2} - 3\sqrt{2}\sqrt[4]{2}}{16} & \textnormal{otherwise}.
    \end{cases}
    \]

    Piecing together \cite[Conj.\ 9.3]{LoughranSantens2024}, we have that the leading constant is given by
    \begin{align*}
    c(\Heis_4) &= 
    \frac{|Z_{S_{64}}(\Heis_{4})||\Heis_4|}{|N_{S_{64}}(\Heis_4)|} \cdot 
    \frac{a^{b-1} \cdot |\Br_C^{e} B\Heis_4|}{|\Hom(\Heis_4,\Q^{\times})| \cdot (b-1)!} 
    \cdot \frac{1}{|\Br_C^{e} B\Heis_4|} \sum_{\beta \in \Br_C^{e} B\Heis_4} \widehat{\tau}(\beta)
    \\ &= \frac{1}{2^{4} \cdot 3^{2}} \cdot \frac{1}{2^{49} \cdot 3^{2} \cdot 5 \cdot 7} 
    \cdot \frac{\widehat{\tau}(0) + 3\widehat{\tau}(\beta_{\neq0})}{4}
    \\ &=
    \frac{1}{2}\Bigg(
        \frac{313+54\sqrt{2}+54\sqrt[4]{2}+9\sqrt{2}\sqrt[4]{2}}{2^{71} \cdot 3^4 \cdot 5} \prod_{p > 2} \left(1 - \frac{1}{p}\right)^{9}\left(1 + \frac{9}{p} + \frac{9}{p^{3/2}}\right) 
        \\ & \qquad\qquad 
        + \frac{247 + 30\sqrt[4]{2} + 36\sqrt{2} - 3\sqrt{2}\sqrt[4]{2}}{2^{71} \cdot 3^3 \cdot 5} \prod_{p > 2} \left(1 - \frac{1}{p}\right)^{9}\left(1 + \frac{9}{p} + \frac{3}{p^{3/2}}\right)
    \Bigg).
    \end{align*}
    This finishes the proof.
\end{proof}

\begin{proof}[Proof of \textnormal{(2)}]
    This follows from the interpretation of the leading constant in \cref{conj:explicit Heis_4 conjecture} sketched in the introduction; we now justify it.
    Let $p_1 \ne p_2$ be odd primes, $\Sigma_{12}$ the local event that both primes ramify, and for each $i = 1,2$ let $\Sigma_i$ be the local event that $p_i$ ramifies, $\varepsilon_{i,0} := (9p_i^{-1}+9p_i^{-3/2})/(1+9p_i^{-1}+9p_i^{-3/2})$, and $\varepsilon_{i,\beta} := (9p_i^{-1}+3p_i^{-3/2})/(1+9p_i^{-1}+3p_i^{-3/2})$.
    Then \cref{thm:intro theorem for Heis_2^n results}(3) and \cite[Conj.\ 9.10]{LoughranSantens2024} imply that
    \[
    \lim_{X \to \infty} \frac{N(\Heis_4,X;\Sigma_{12})}{N(\Heis_4,X)} = 
    \frac{C_0 \varepsilon_{1,0} \varepsilon_{2,0} +
    C_{\beta} \varepsilon_{1,\beta} \varepsilon_{2,\beta}}
    {C_0 + C_{\beta}},
    \]
    as well as
    \[
    \left(\lim_{X \to \infty} \frac{N(\Heis_4,X;\Sigma_{1})}{N(\Heis_4,X)}\right)
    \left(\lim_{X \to \infty} \frac{N(\Heis_4,X;\Sigma_{2})}{N(\Heis_4,X)}\right) =
    \frac{C_0 \varepsilon_{1,0} + C_{\beta} \varepsilon_{1,\beta}}{C_0 + C_{\beta}}
    \frac{C_0 \varepsilon_{2,0} + C_{\beta} \varepsilon_{2,\beta}}{C_0 + C_{\beta}}.
    \]
    One can directly show that the difference between the two quantities is
    \[
    \frac{36 C_0 C_{\beta}}{(C_0 + C_{\beta})^2} \frac{1}{p_1^{3/2} p_2^{3/2}}
    \left(1 + O\left(\frac{1}{p_1} + \frac{1}{p_2}\right)\right).
    \]
    This is nonzero for all $p_1 \ne p_2$ sufficiently large, demonstrating that $\Heis_4 \le S_{64}$ fails to satisfy almost-everywhere independence of local events.
\end{proof}

\subsection{Algebraic Brauer groups of non-concentrated permutation groups}

We note that the non-concentrated permutation group $\Heis_4 \le S_{64}$ appearing in \cref{conj:explicit Heis_4 conjecture} has a \textit{transcendental} partially unramified Brauer--Manin obstruction, rather than an algebraic one.
This is not a coincidence, as we now show that all non-concentrated permutation groups have trivial partially unramified algebraic Brauer groups \cite[Def.\ 6.10, \S6.6]{LoughranSantensSurvey}.

Recall \cite[Def.~6.10]{LoughranSantensSurvey} that the Brauer group depends on a conjugacy invariant and Galois invariant subset $C \subseteq G(-1)$, where $G(-1)$ is the $(-1)$-Tate twist (see \cref{def:Tate_twist}).

\begin{lemma}
    \label{lem:vanishing of algebraic Brauer groups in many cases}
	Let $k$ be a field and $G$ a finite group of order coprime to the characteristic of $k$.
    Let $C \subset G(-1)_k$ be a conjugacy invariant and Galois invariant subset that generates $G$ and whose nontrivial elements all have prime order.
    Then $\Br^e_{1, C} (BG)_k = 0$.
\end{lemma}

\begin{proof}
    In this proof, all $(-1)$-Tate twists occur over the base field $k$.
	Let $\pi: G \twoheadrightarrow G^{\mathrm{ab}}$ be the abelianization map and $C^{\mathrm{ab}} := \pi(C) \subset G^{\mathrm{ab}}(-1)$.
    We have that $C^{\rm ab}$ generates $G^{\mathrm{ab}}$, since $C$ generates $G$ (both in the sense of \cite[Def.\ 3.4]{LoughranSantens2024}).

    We claim that the subset of elements $C^{\rm ab} \subset G^{\rm ab}(-1)$ is \textit{fair} \cite{WoodAbelianLocalProbabilities}, meaning for every $r \in \N$ one has that $C^{\rm ab} \cap G^{\rm ab}[r](-1)$ generates $G^{\rm ab}[r]$.
    Indeed, since $G^{\rm ab}$ is generated by elements of prime order, we deduce that $G^{\rm ab} = \bigoplus_p G^{\rm ab}[p]$, hence for all $r\in \N$ we have that
    \[
    \langle C^{\rm ab} \cap G^{\rm ab}[r](-1)\rangle = \bigoplus_{p \mid r} \langle C^{\rm ab} \cap G^{\rm ab}[p](-1) \rangle = \bigoplus_{p \mid r} G^{\rm ab}[p] = G^{\rm ab}[r].
    \]

    Since the subset $C^{\rm ab} \subset G^{\rm ab}(-1)$ is fair, we deduce that $\Br^e_{1, C^{\mathrm{ab}}} (BG^{\mathrm{ab}})_k = 0$ using \cite[Lem.~10.26]{LoughranSantens2024}, \cite[Lem.~10.21]{LoughranSantens2024}, and the observation that for every prime $p$ one has that $\Sha^1_{\omega}(k, \mu_p) \subset \break \operatorname{H}^1(\Gal(k(\mu_p)/k), \mu_p)= 0$.
    The map $\Br^e_{1, C^{\mathrm{ab}}} (BG^{\mathrm{ab}})_k \to \Br^e_{1, C} (BG)_k$ is an isomorphism by \cite[Lem.~6.31]{LoughranSantens2024},\footnote{\cite[Lem.~6.31]{LoughranSantens2024} technically provides an isomorphism $\Br^e_{1,C^{\rm ab}}(BG^{\rm ab})_k \to \Br^e_{1, f^{-1}(C^{\mathrm{ab}})} (BG)_k$, but the argument given there also works for the map $\Br^e_{1, C^{\mathrm{ab}}} (BG^{\mathrm{ab}})_k \to \Br^e_{1, C} (BG)_k$.} proving the claim.
\end{proof}

\begin{corollary}
	Let $G \le S_n$ be a transitive permutation group and $C \subset G$ the subset of minimal index elements. Assume that $C$ generates $G$. Then $\Br^e_{1, C} BG = 0$.
\end{corollary}

\begin{proof}
	The elements of minimal index have prime order, since if $g \in G \backslash \{1\}$ has composite order $n$, and $p \mid n$ is prime, then the index of $g^p \ne 1$ is strictly smaller than that of $g$.
\end{proof}

\begin{remark}
    The assumption that $C$ generates $G$ is crucial.
    For example, if $G = \Z/8\Z$ then $\langle C \rangle = 4\Z/8\Z \ne G$ and Wang's counterexample \cite{GrunwaldWang1} to a result of Grunwald \cite{GrunwaldWang2} gives a nonzero Brauer class $\beta_{16} \in \Br_{1,{\rm un}}^{e} B \Z/8\Z \subset \Br_{1,C}^{e} B\Z/8\Z$; see \cite[Ex.\ 6.22]{LoughranSantensSurvey}.
\end{remark}

\subsection{An algorithm for other $2$-groups}

The statement and proof of \cref{thm:derivation from Loughran Santens conjecture}, which explicitly compute the predicted leading constant in Malle's conjecture for the regular permutation group $\Heis_4$ and describe the statistics of local events, naturally generalize to finite $2$-groups $G$ of nilpotency class $\le 2$ and exponent $\le 4$ in their regular representation.
Indeed, the statement and proof of \cref{thm:derivation from Loughran Santens conjecture} for a general such $G$ require the computation of the following objects:
\begin{enumerate}
    \item the partially unramified Brauer group $\Br_{C}^{e} BG$;
    \item the residues of all $\beta \in \Br_C^e BG$;
    \item the enumeration of $\Hom(\Gamma_{\Q_2},G)$ with Galois discriminants;
    \item the Brauer--Manin pairing $\Hom(\Gamma_{\Q_2},G) \times \Br_C^e BG \to \frac{1}{2}\Z/\Z$.
\end{enumerate}
One also needs to compute the $\infty$-adic Brauer transforms of $G$; however, these are readily computable by \cref{lem:archimedean local Brauer transforms for general groups}.

Item (1) is computable in terms of quadratic forms using \cref{lem:partially_ramified_Brauer_group}.
Item (2) is computable in terms of quadratic forms by adapting the statement and proof of \cref{lem:residue_computation}.
Item (3) is computable in terms of $G[2]$ and \textit{admissible triples} of $G$; see \cref{def:admissible triple,thm:bijection between 2-adic homomorphisms and admissible triples}.
Item (4) is computable in terms of quadratic forms and admissible triples by adapting the statement and proof of \cref{cor:2-adic mass in terms of admissible triples}.
This verifies that all steps in writing down the leading constant in Malle's conjecture for $G$ are explicitly computable in terms of quadratic forms and admissible triples.

\section{Two-step nilpotent groups}
\label{sec:two step nilpotent groups}

\subsection{Brauer groups}
Let us recall the following definitions from \cite{LoughranSantensSurvey}.
For the sake of concreteness, we always take our base number field to be $\Q$.

A \textit{$\Gamma_{\Q}$-set} or simply \textit{Galois set} is a set $X$ equipped with a (left) $\Gamma_{\Q}$-action.
A subset $Y\subset X$ is \textit{$\Gamma_{\Q}$-invariant} or \textit{Galois invariant} if $Y$ is preserved by the $\Gamma_{\Q}$-action.

A \textit{$\Gamma_{\Q}$-group} is a group $G$ equipped with a (left) $\Gamma_{\Q}$-action, and a \textit{$\Gamma_{\Q}$-subgroup} $H\le G$ is a subgroup of $G$ that is preserved by the $\Gamma_{\Q}$-action.
If $G_1,G_2$ are $\Gamma_{\Q}$-groups, we let $\Hom(G_1,G_2)$ denote the Galois set of group homomorphisms $\varphi:G_1 \to G_2$, with $\Gamma_{\Q}$-action given by
\[
\forall \sigma \in \Gamma_{\Q}, \forall g \in G_1, \qquad 
(\sigma.\varphi)(g) := \sigma.(\varphi(\sigma^{-1}.g)).
\]
Let
$\Hom_{\Q}(G_1,G_2) := \Hom(G_1,G_2)^{\Gamma_{\Q}}$ denote the set of $\Gamma_{\Q}$-equivariant group homomorphisms.
The elements of $\Hom_{\Q}(G_1,G_2)$ correspond to morphisms $G_1 \to G_2$ in the category of $\Gamma_{\Q}$-groups.

\begin{remark}
Note that phrases such as ``a sequence of $\Gamma_{\Q}$-groups'' or ``an extension of $\Gamma_{\Q}$-groups'' are meant to be interpreted in the category of $\Gamma_{\Q}$-groups.
\end{remark}

The multiplicative group $\Qbar^{\times}$ is a $\Gamma_{\Q}$-group, and the subset of \textit{roots of unity} $\mu_{\infty} \le \Qbar^{\times}$ is a $\Gamma_{\Q}$-subgroup.
We also define the $\Gamma_{\Q}$-group $\widehat{\Z}(1) := \varprojlim_n \mu_n$ as a limit in the category of $\Gamma_{\Q}$-groups.
Observe that $\widehat{\Z}(1)$ is a pro-cyclic group, hence it admits a topological generator.

\begin{definition}
    Let $G$ be a $\Gamma_{\Q}$-group.
	A \emph{Brauer element of $BG$} or simply \textit{Brauer element} is a central extension of $\Gamma_{\Q}$-groups
    \[
    \beta: \quad 
    1 \to \mu_{\infty} \to G_{\beta} \to G \to 1.
    \]
    We define the \textit{(normalized) Brauer group} $\Br^e BG$ to be the set of isomorphism classes of Brauer elements of $BG$, equipped with a group structure given by Baer sums of central extensions (see \cite[Lem.\ 6.2]{LoughranSantensSurvey}).
\end{definition}

\begin{remark}
When making predictions for the leading constant in Malle's conjecture, it is sufficient to study the normalized Brauer group $\Br^e BG$ instead of the full Brauer group $\Br\, BG$ (see \cite[Rem.\ 6.3]{LoughranSantensSurvey}).
\end{remark}

\begin{definition}
    \label{def:Tate_twist}
	Let $G$ be a $\Gamma_{\Q}$-group.
    The \textit{$(-1)$-Tate twist} of $G$ is the Galois set $G(-1) := \Hom(\widehat{\Z}(1), G)$.
    Note that $G$ acts on $G(-1)$ by (left) conjugation, i.e.\
    \[
    \forall g \in G, \forall h \in G(-1), \forall \zeta \in \widehat{\Z}(1), \qquad 
    (g.h)(\zeta) := gh(\zeta)g^{-1}.
    \]
    We say that a subset $Y \subset G(-1)$ is \textit{conjugacy invariant} if $Y$ is preserved by the $G$-action above.

	Let $C \subset G(-1) \setminus \{1\}$ be a subset that is conjugacy invariant and Galois invariant.
    A \emph{$C$-marking} of a Brauer element $\beta: 1 \to \mu_{\infty} \to G_{\beta} \to G \to 1$ is a conjugacy invariant and Galois invariant subset $M \subset G_{\beta}(-1)$ such that the map $G_{\beta}(-1)  \to G(-1)$ induces a bijection $M \xrightarrow{\sim} C$.
	
	Let $\Br^e_C BG \le \Br^{e} BG$ be the subgroup of Brauer elements which admit a $C$-marking (this subset is indeed a subgroup; see \cite[\S6.3]{LoughranSantensSurvey}).
    We refer to $\Br^e_C BG$ as a Brauer group.
\end{definition}

\begin{remark}
    For any $\Gamma_{\Q}$-group $G$, an element $h \in G(-1)$ can be \textit{powered}, i.e.\ for all $m\in\Z$ one has an element $h^m \in G(-1)$ via $h^m(\zeta) := (h(\zeta))^m$ for all $\zeta \in \widehat{\Z}(1)$.
\end{remark}

\begin{remark}
\label{rem:order 2 elements abuse of notation}
Keeping track of the Galois action on $C$ is often inconvenient, hence we abuse notation as follows.
If $C \subset G$ consists of the order $2$ elements of $G$, then $C$ can be canonically identified with the elements of $G(-1)$ sending any topological generator of $\widehat{\Z}(1)$ to an order $2$ element of $G$.
For this particular choice of $C \subset G$, the following lemma shows that the group $\Br_C^e BG$ can often be completely described using the category of abstract groups.
\end{remark}

\begin{lemma}
    \label{lem:Brauer_group_discriminant}
	Let $G$ be an abstract group and $C \subset G$ the subset of order $2$ elements.
    Assume that $C$ generates $G$.
	The Brauer group $\Br^{e}_C BG$ can be identified with the group of isomorphism classes of central extensions of abstract groups
	\[
	0 \to \F_2 \to \widetilde{G} \to G \to 1
	\]
	satisfying the following properties:
	\begin{enumerate} 
		\item Every element in the pre-image of $C$ in $\widetilde{G}$ has order $2$;
		\item For every conjugacy class $c \subset C$, the pre-image of $c$ consists of two conjugacy classes.
	\end{enumerate}
\end{lemma}

\begin{remark}
	The relevance to Malle's conjecture is that, if $C$ generates $G$, then $\Br^{e}_C BG$ is the Brauer group appearing in the leading constant in Malle's conjecture for Galois $G$-number fields ordered by discriminant.
    See \cite[Conj.\ 9.3]{LoughranSantens2024}, \cite[Conj.\ 7.1]{LoughranSantensSurvey}.
\end{remark}

\begin{proof}
    Let $\mc{E}(G)$ denote the set of isomorphism classes of central extensions of abstract groups appearing in the lemma statement, and $\kappa: 0 \to \F_2 \to \widetilde{G} \to G \to 1$ a representative of an element $[\kappa] \in \mc{E}(G)$.
    $\mc{E}(G)$ will be shown to be closed under Baer sum, hence $\mc{E}(G)$ will be a group, because we will establish a bijection from $\mc{E}(G)$ to $\Br_C^e BG$ with sufficiently nice properties.
    
    We may view $\kappa$ as a central extension of $\Gamma_{\Q}$-groups with trivial $\Gamma_{\Q}$-actions.
    There is also a canonical isomorphism of $\Gamma_{\Q}$-groups $\F_2 \cong \mu_2$, and this allows us to consider the morphism of short exact sequences in the category of $\Gamma_{\Q}$-groups given by
    \[
    \begin{tikzcd}
    \kappa:         & 1 \arrow[r] & \mu_2 \arrow[r] \arrow[d] & \widetilde{G} \arrow[r] \arrow[d] & G \arrow[r] \arrow[d] & 1 \\
    \beta_{\kappa}: & 1 \arrow[r] & \mu_{\infty} \arrow[r]    & G_{\beta} \arrow[r]               & G \arrow[r]           & 1,
    \end{tikzcd}
    \]
    where the left square is a pushout diagram.
    Moreover, $\beta_{\kappa} : 1 \to \mu_{\infty} \to G_{\beta} \to G \to 1$ is a Brauer element because pushouts preserve central extensions of $\Gamma_{\Q}$-groups.
    Thus, we have a candidate for a map $\mc{E}(G) \to \Br^e BG$ given by $[\kappa] \mapsto \beta_{\kappa}$, where $\beta_{\kappa}$ denotes the isomorphism class of the displayed Brauer element.

    We check that $\mc{E}(G) \to \Br_C^e BG$ is a well-defined map (and will be seen to be a homomorphism).
	Pushouts are compatible with taking isomorphism classes (and will be seen to intertwine Baer sums), therefore it remains to check that $\beta_{\kappa}$ admits a $C$-marking.
    The conditions on $\kappa$ imply that there is a conjugacy invariant subset $M \subset \widetilde{G}$ that maps isomorphically onto $C$.
    This set $M$ consists of order $2$ elements, and therefore by \cref{rem:order 2 elements abuse of notation} can be identified with a Galois invariant subset $M \subset \widetilde{G}(-1)$. 
    This subset defines a $C$-marking of $\beta_{\kappa}$ via $M \subset \widetilde{G}(-1) \hookrightarrow G_{\beta}(-1)$.
	
	We check that $\mc{E}(G) \to \Br_C^e BG$ is injective.
    Since $\mu_{2} \to \mu_{\infty}$ is injective, we may identify $\widetilde{G}$ with a subset of $G_{\beta}^{\Gamma_{\Q}}$.
    On the other hand, taking $\Gamma_{\Q}$-invariants of $\beta_{\kappa}$ implies that we have an exact sequence of abstract groups
	\[
	1 \to \mu_2 \to G_{\beta}^{\Gamma_{\Q}} \to G, 
	\]
	therefore we find that $\widetilde{G} = G_{\beta}^{\Gamma_{\Q}}$.
    Therefore, we recover $[\kappa]$ from $\beta_{\kappa}$ by taking $\Gamma_{\Q}$-invariants, so the map is injective.
	
	We check that $\mc{E}(G) \to \Br_C^e BG$ is surjective.
    The previous paragraph suggests that if $\beta \in \Br^e_C BG$, then taking $\Gamma_{\Q}$-invariants of $\beta$ will recover an element of $\mc{E}(G)$ mapping to $\beta$; it remains to show that this actually produces a central extension of G, and that its isomorphism class belongs to $\mc{E}(G)$.
    We do this in two steps.

	\textit{Step 1.} We show that $G_{\beta}^{\Gamma_{\Q}} \to G$ is surjective.
    Fix a $C$-marking $M \subset G_{\beta}(-1)$.
    For each $g \in C$, let $g_{\beta} \in M$ be the corresponding marked element.
    Since $g$ has order $2$ (\cref{rem:order 2 elements abuse of notation}), the element $z_g := g_{\beta}^2$ belongs to $\mu_{\infty}(-1)$.
    Moreover, there is a canonical isomorphism of Galois sets $\mu_{\infty}(-1) \cong \Q/\Z$.
    Since $\Q/\Z$ is a divisible group, we may therefore make sense of $z_g^{-1/2}$ and define
    \[
    \widetilde{g} := z_g^{-1/2} g_{\beta}.
    \]
    Note that $\widetilde{g} \in G_{\beta}(-1)$ has order $2$, and can thus be identified with an element of $G_{\beta}$.
	The Galois action on $C$ is trivial, and this implies that the Galois action on $g_{\beta}$, and thus also on $\widetilde{g}$, is trivial; therefore, $\widetilde{g} \in G_{\beta}^{\Gamma_{\Q}}$. 
    It follows that $g$ lies in the image of $G_{\beta}^{\Gamma_{\Q}} \to G$.
    As this holds for all $g \in C$, and $C$ generates $G$ by assumption, we deduce that $G_{\beta}^{\Gamma_{\Q}} \to G$ is surjective.
	
    \textit{Step 2.} 
    We verify the remaining properties needed for this isomorphism class of central extensions to belong to $\mc{E}(G)$.
	Property (1) holds for $\widetilde{G} := G_{\beta}^{\Gamma_{\Q}}$, as $\widetilde{g} \in \widetilde{G}$ has order $2$ for all $g\in C$ by construction.
    For property (2), assume for the sake of contradiction that there exists an $h \in \widetilde{G}$ such that $h \widetilde{g} h^{-1} = \widetilde{g} w$, where $w \in \F_2$ is the nontrivial element.
    Then in $G_{\beta}(-1)$ one would obtain the relation $h g_{\beta} h^{-1} = g_{\beta} w$, since $z_g^{-1/2}$ is central.
    Both $g_{\beta}$ and $g_{\beta} w$ are elements of the marking $M \subset G_{\beta}(-1)$, since $M$ is conjugacy invariant, and yet both $g_{\beta}$ and $g_{\beta} w$ map to $g \in C$.
    Therefore, the map $M \to C$ fails to be injective, which is a contradiction.
\end{proof}

\subsection{Two-step nilpotent groups}
\label{subsec:two-step nilpotent groups}

We recall the following classification of two-step nilpotent groups.
\begin{lemma}\label{lem:central_extensions_quadratic_forms}
	Let $0 \to W \to G \to V \to 0$ be a central extension where $V, W$ are finite-dimensional vector spaces over $\F_2$. The map $Q_{G}: V \to W$ which sends $v \in V$ to the square $\widetilde{v}^2$ of a lift $\widetilde{v} \in G$ of $v$ is a quadratic map.
	
	Moreover, the assignment $G \mapsto Q_G$ defines a bijection between the group of isomorphism classes of central extensions $0 \to W \to G \to V \to 0$ and the vector space of quadratic maps $V \to W$.
\end{lemma}

\begin{proof}
	This is \cite[Cor.~2.4]{Pakianathan2011Quadratic}.
\end{proof}

We will use the following computation.

\begin{lemma}\label{lem:commutator_lifts}
	Let $v, w \in V$ and let $\widetilde{v}, \widetilde{w}$ be respective lifts. Then $[\widetilde{v}, \widetilde{w}] = Q_G(v + w) - Q_G(v) - Q_G(w) \in W$. In particular, $\widetilde{v}$ and $\widetilde{w}$ commute if and only if $Q_G(v + w) = Q_G(v) + Q_G(w)$.
\end{lemma}

\begin{proof}
	This is immediate from the centrality of $W$ and the fact that $\widetilde{v} = \widetilde{v}^{-1} Q_G(v), \widetilde{w} = \widetilde{w}^{-1} Q_G(w), \widetilde{v} \widetilde{w} = \widetilde{w}^{-1} \widetilde{v}^{-1} Q_G(v + w)$ by the definition of $Q_G$.
\end{proof}

We can combine this description of two-step nilpotent groups with \cref{lem:Brauer_group_discriminant} to give a complete description of the Brauer group of such a central extension. Let us first note the following.

\begin{lemma}\label{lem:2-torsion elements}
	Let $0 \to W \to G \to V \to 0$ be a central extension as in \cref{lem:central_extensions_quadratic_forms}.
    The order $2$ elements of $G$ are the non-trivial elements of $G$ whose image $v \in V$ is such that $Q_G(v) = 0$.
\end{lemma}

\begin{proof}
	This is immediate from the definition of $Q_G$.
\end{proof}

We now fix finite-dimensional $\F_2$-vector spaces $V$ and $W$, a central extension $0 \to W \to G \to V \to 0$.
For every quadratic form $Q_{\beta}: V \to \F_2$, we will denote the central extension corresponding to the quadratic map $(Q_G, Q_{\beta}): V \to W \oplus \F_2$ in \cref{lem:central_extensions_quadratic_forms} by
\[
	0 \to W \oplus \F_2 \to G_{\beta} \to V \to 0.
\]
	This defines a central extension $0 \to \F_2 \to G_{\beta} \to G \to 1$ by a simple diagram chase, which in turn induces a Brauer class $\beta \in \Br_C^e BG$ by \cref{lem:Brauer_group_discriminant}, where $C\subset G$ is the subset of order $2$ elements.
Letting $\operatorname{Quad}(V)$ denote the vector space of quadratic forms $V \to \F_2$, we have constructed a set-theoretic map
\[
\psi_G : \operatorname{Quad}(V) \to \Br^e BG.
\]
It is straightforward but tedious to check that $\psi_G$ is a group homomorphism.

We have a linear map $Q_G^* : W^\vee \to \operatorname{Quad}(V)$ given by $\alpha \mapsto \alpha \circ Q_G$.
We write $Q_{\beta + \alpha} : V \to \F_2$ for the quadratic map $Q_{\beta} + Q_G^* \alpha$.
Note that the map $W \oplus \F_2 \to W \oplus \F_2$ given by $w \oplus x \mapsto w \oplus (x + \alpha(w))$ induces an isomorphism between the quadratic maps $(Q_G, Q_{\beta})$ and $(Q_G, Q_{\beta + \alpha})$.
By \cref{lem:central_extensions_quadratic_forms}, this induces an isomorphism $G_{\beta} \cong G_{\beta + \alpha}$ that preserves $0 \oplus \F_2$, and thus induces a morphism of central extensions from $0 \to \F_2 \to G_{\beta} \to G \to 1$ to $0 \to \F_2 \to G_{\beta + \alpha} \to G \to 1$.
In particular, we have shown that $\psi_G$ factors through the quotient map $\operatorname{Quad}(V) \twoheadrightarrow \operatorname{Quad}(V)/Q_G^*W^\vee$.
As an abuse of notation, we write $\operatorname{Quad}(V)/W^\vee$ in place of $\operatorname{Quad}(V)/Q_G^*W^\vee$, and write the factorization of $\psi_G$ through the above quotient map as a group homomorphism
\[
\Psi_G : \operatorname{Quad}(V)/W^\vee \to \Br^e BG.
\]

\begin{lemma}
    Let $V,W$ be finite-dimensional vector spaces over $\F_2$, $0 \to W \to G \to V \to 0$ a central extension, and $C \subset G$ be the subset of order $2$ elements.
	The homomorphism $\Psi_G : \operatorname{Quad}(V)/W^{\vee} \to \Br^{e} BG$ is injective, and its image contains $\Br^{e}_C BG$.
\end{lemma}

\begin{proof}
	Suppose that $Q_{\beta} \in \operatorname{Quad}(V)$ lies in the kernel of $\psi_G$.
    Then there exists a group-theoretic section $s: G \cong G_{\beta}/\F_2 \to G_{\beta}$.
    Restricting to $W$ defines a section $s|_W: W \to W \oplus \F_2$. The second coordinate is a linear map $\alpha$.
	
	Let $v \in V$ and $\widetilde{v} \in G$ be a lift.
    Then $s(\widetilde{v}) \in G_{\beta}$ is a lift of $v$ and we compute
    \[
    (Q_G,Q_{\beta})(v) =
    s(\widetilde{v})^2 = s(\widetilde{v}^2) = s(Q_G(v)) = (Q_{G}(v), (Q_G^* \alpha)(v))
    \]
    This shows that $Q_{\beta} = Q_G^* \alpha$, hence $\Psi_G$ is injective.
	
	We now show that $\Br^{e}_C BG$ is contained in the image of $\Psi_G$.
    By \cref{lem:Brauer_group_discriminant}, an element of $\Br^{e}_C BG$ can be represented by a central extension $0 \to \F_2 \to \widetilde{G} \xrightarrow{\pi} G \to 1$.
    Let $\widetilde{W} := \pi^{-1}(W)$.
    The elements in $W$ are $2$-torsion and central, so \cref{lem:Brauer_group_discriminant} implies that the elements in $\widetilde{W}$ are $2$-torsion and central.
	Thus, we have a central extension $1 \to \widetilde{W} \to \widetilde{G} \to V \to 1$.
    By \cref{lem:central_extensions_quadratic_forms}, this corresponds to a quadratic map $q_{\widetilde{G}}: V \to \widetilde{W}$. Moreover, one directly computes that $\pi \circ q_{\widetilde{G}} = q_G$. 
	
	The group $\widetilde{W}$ is an $\F_2$-vector space, so the homomorphism $\pi: \widetilde{W} \to W$ splits.
    Choosing a splitting $\widetilde{W} \cong W \oplus \F_2$, we find a quadratic form $Q_{\beta}: V \to \F_2$ such that $Q_{\widetilde{G}} = (Q_{G}, Q_{\beta})$.
    This finishes the lemma.
\end{proof}

We also determine the pre-image of $\Br^e_C BG$ under $\Psi_G$.

\begin{lemma}
    \label{lem:partially_ramified_Brauer_group}
	The pre-image of $\Br^e_C BG$ under the homomorphism $\Psi_G:\operatorname{Quad}(V)/W^{\vee} \to \Br^{e} BG$ is given by the set of quadratic forms $Q_{\beta}: V \to \F_2$ with the following properties:
	\begin{enumerate}
	\item For all $v \in V$ such that $Q_{G}(v) = 0$, we have that $Q_{\beta}(v) = 0$.
	\item For all $v, w \in V$ such that $Q_{G}(v) = 0$ and $Q_{G}(w) = Q_{G}(v + w)$, we have that $Q_{\beta}(w) = Q_{\beta}(w + v)$.
	\end{enumerate}
\end{lemma}

\begin{remark}
	The second property can be rephrased in terms of the alternating form $\langle v, w \rangle_{\bullet} := Q_{\bullet}(v + w) - Q_{\bullet}(v) - Q_{\bullet}(w)$ as follows: if $v, w \in V$ are such that $Q_{G}(v) = 0$ and $\langle v, w \rangle_{G} = 0$, then $\langle v, w \rangle_{\beta} = 0$.
\end{remark}
\begin{proof}
	This is immediate by the construction of $Q_{G}$ and \cref{lem:Brauer_group_discriminant,lem:2-torsion elements}.
\end{proof}

\section{The case of Heisenberg groups}
\label{sec:the case of Heisenberg groups}

\begin{notation}
    Let $n \in \Z_{\ge 1}$.
    For calculations, we often identify $\Heis_{2^n}$ with $\F_{2^n}^3$ via the coordinates $(x,y,z)$ appearing in \cref{def:Heisenberg group}, and note that the multiplication law takes the shape
    \[
    (x,y,z)(x',y',z') = (x+x',y+y',z+z'+xy').
    \]
    The multiplication $xy'$ is the field multiplication in $\F_{2^n}$.
\end{notation}

\begin{lemma}\label{lem:quadratic_map_Heis}
	We have a central extension $1 \to \F_{2^n} \to \Heis_{2^n} \to \F_{2^n} \oplus \F_{2^n} \to 1$ with corresponding quadratic form $Q_{\Heis_{2^n}}:\F_{2^n} \oplus \F_{2^n} \to \F_{2^n}$, $(x, y) \mapsto xy$. 
\end{lemma}

\begin{proof}
	This follows from a direct computation.
\end{proof}

Let us determine the conjugacy classes of $\Heis_{2^n}$.
\begin{lemma}
\label{lem:conjugacy_classes}
	The group $\Heis_{2^n}$ has three types of non-trivial conjugacy classes.
	\begin{enumerate}
		\item It has $2^n - 1$ central conjugacy classes which contain $1$ element and are of the form $(0,0,z)$ for some $z \in \F_{2^n}^{\times}$.
		\item It has $2(2^n - 1)$ non-central conjugacy classes of order $2$; these have size $2^n$ and are of the form $\{(x, 0, z): z \in \F_{2^n}\}$ or $\{(0, x, z): z \in \F_{2^n}\}$ for some fixed element $x \in \F_{2^n}^{\times}$.
		\item It has $(2^n - 1)^2$ non-central conjugacy classes of order $4$. These have size $2^n$ and are of the form $\{(x, y, z): z \in \F_{2^n}\}$ for some fixed elements $x, y \in \F_{2^n}^{\times}$.
	\end{enumerate}
\end{lemma}

\begin{proof}
	This is well-known and follows from a direct computation using \cref{lem:commutator_lifts}.
\end{proof}

\begin{lemma}\label{lem:Brauer_group_bilinear_pairing}
    Let $C \subset \Heis_{2^n}$ be the subset of order $2$ elements.
	The subgroup of $\operatorname{Quad}(\F_{2^n} \oplus \F_{2^n})/\F_{2^n}^{\vee}$ corresponding to $\Br_{C}^e B\Heis_{2^n}$ is given by all quadratic maps of the form 
	\[
	Q_{b}: \F_{2^n} \oplus \F_{2^n} \to \F_2, \quad (x, y) \mapsto b(x, y),
	\]
	where $b$ is an $\F_2$-bilinear form $b: \F_{2^n} \times \F_{2^n} \to \F_2$. 
	
	In particular, $|\Br_{C}^e B\Heis_{2^n}| = 2^{n(n-1)}$.
\end{lemma}

\begin{proof}
	By \cref{lem:partially_ramified_Brauer_group,lem:2-torsion elements,lem:commutator_lifts}, the relevant quadratic forms $Q$ are exactly the ones with the following properties:
	\begin{enumerate}
		\item $Q(x,0) = Q(0,y)$ for all $x, y \in \F_{2^n}$;
		\item For all $x,y,u \in \F_{2^n}$ such that $xy= (x + u)y$ (respectively $xy = x(y + u)$), we have that $Q(x,y) = Q(x + u, y)$ (respectively $Q(x,y) = Q(x, y + u)$).
	\end{enumerate}
	
	Note that the second condition is actually implied by the first condition, since $xy= (x + u)y$ (respectively $xy = x(y + u)$) holds only if $u = 0$ or $y = 0$ (respectively $x = 0$).
    The lemma then follows from a direct computation after choosing an $\F_2$-basis of $\F_{2^n}$.
	
	One can directly check that a quadratic form given by $Q_{b}$ above satisfies the first (and hence second) condition.
    Conversely, if a quadratic form $Q$ satisfies the first condition, then the polarization
    \[
    b_Q(x,y) := Q(x,y) - Q(x, 0) - Q(0, y)
    \]
    satisfies $b_Q(x,y) = Q(x,y)$.
    It follows by definition that $Q = Q_{b_Q}$, finishing the proof.
\end{proof}

\subsection{Residues}
We now compute the \textit{residues} of the Brauer elements we determined in \cref{lem:Brauer_group_bilinear_pairing}. 
For the definition of the residue, see \cite[Def.\ 6.14]{LoughranSantensSurvey}.

We use the classification of conjugacy classes from \cref{lem:conjugacy_classes}.
Note that all conjugacy classes are fixed under invertible powers, and can thus be canonically identified with conjugacy classes of $G(-1)$.

For $G = \Heis_{2^n}$, let $C \subset G$ be the subset of order $2$ elements.
The residues at the order $2$ conjugacy classes are trivial by tautology: $\beta \in \Br^e B\Heis_{2^n}$ belongs to $\Br_{C}^{e} B\Heis_{2^n}$ if and only if the residues of $\beta$ along $C$ are trivial \cite[Lem.~6.15]{LoughranSantensSurvey}.
Thus, it only remains to compute the residues at the third type of conjugacy class.

\begin{lemma}\label{lem:residue_computation}
	Let $n \in \Z_{\ge 1}$, $G = \Heis_{2^n}$, and $C \subset G$ be the subset of order $2$ elements.
    Suppose $\beta \in \Br_C^e B\Heis_{2^n}$ corresponds to an $\F_2$-bilinear form $b:\F_{2^n} \times \F_{2^n} \to \F_2$.
    Given $x,y \in \F_{2^n}^{\times}$, set $c := \{(x,y,z): z \in \F_{2^n} \}$.
    The residue of $\beta$ at $c$ is algebraic if and only if $b(\lambda x,y) = b(x, \lambda y)$ for all $\lambda \in \F_{2^n}$.
    Moreover, every algebraic residue is trivial.
\end{lemma}

\begin{proof}
	Let $\widetilde{g} \in G_{\beta}$ be a lift of $(x,y,0) \in c$, and $\widetilde{c} \subset G_{\beta}$ denote the conjugacy class of $\widetilde{g}$.
    The residue of $\beta$ at $c$ is algebraic if and only if the map $\widetilde{c} \to c$ is a bijection.
    By the orbit stabilizer theorem, $\widetilde{c} \to c$ is a bijection if and only if the map of centralizers $\pi : Z(\widetilde{g}) \to Z((x,y,0))$ is a surjection (as its kernel always contains the central $\F_2$ appearing in $0 \to \F_2 \to G_{\beta} \to G \to 1$).
	
    Notice that $\pi$ maps $W\oplus\F_2$ onto $W$, therefore $\pi$ is surjective if and only if the $\F_2$-linear map $\pi_V : Z_V(\widetilde{g}) \to Z_V((x,y,0))$ is surjective, where $Z_V(\widetilde{g}) \le V$ and $Z_V((x,y,0)) \le V$ denote the images of $Z(\widetilde{g}), Z((x,y,0))$ modulo $W\oplus \F_2, W$ respectively.
    A direct computation using \cref{lem:commutator_lifts} shows that 
	\begin{equation*}
		\begin{split}
		&Z_V((x,y,0)) = \{(\lambda x, \lambda y) \in \F_{2^n} \oplus \F_{2^n}: \lambda \in \F_{2^n}\}, \\
		&Z_V(\widetilde{g}) = \{(\lambda x, \lambda y) \in Z_V((x,y,0)): b(\lambda x,y) = b(x, \lambda y) \}.
		\end{split}
	\end{equation*}
	The first claim of the lemma follows.
	
    It remains to show that every algebraic residue is trivial.
    The residue of $\beta$ at $c$ is trivial if and only if $\widetilde{g}$ is conjugate to $\widetilde{g}^{-1} = \widetilde{g}^3$.
	By a direct computation using \cref{lem:commutator_lifts}, the element $\widetilde{g}$ is conjugate to $\widetilde{g}^3$ if and only if there exists $(u,v) \in \F_{2^n} \times \F_{2^n}$ such that $(xy, b(x,y)) \in \F_{2^n} \times \F_2$ is equal to  
	\[
		((x + u)(y + v) - xy - uv, b(x + u, y + v) - b(x,y) - b(u,v)) = (xv + yu, b(x,v) + b(u,y)).
	\]
    The solution set of the first equation $xy = xv + yu$ is precisely $(u,v) = ((1 + \lambda)x, \lambda y)$ for all $\lambda \in \F_{2^n}$, and the second equation is now expressible as
	\begin{equation*}
		b(x,y) = b(x,v) + b(u,y)  = b(x, \lambda y) + b(x,y) + b(\lambda x,y).
	\end{equation*}
    Thus, $\beta$ has trivial residue at $c$ if and only if the above equation has a solution; however, the existence of a solution follows from $\beta$ having an algebraic residue at $c$.
    Therefore, every algebraic residue is trivial.
\end{proof}

\begin{corollary}\label{cor:unramified_Brauer_group}
	The unramified Brauer group $\Br_{\mathrm{un}}^e B\Heis_{2^n}$ is trivial.
\end{corollary}
\begin{proof}
	Let $C \subset \Heis_{2^n}$ be the subset of order $2$ elements.
    Choose $\beta \in \Br_{\mathrm{un}}^e B\Heis_{2^n} \subset \Br_{C}^e B\Heis_{2^n}$ with corresponding $\F_2$-bilinear form $b$.
    Then all residues of $\beta$ are trivial by \cite[Lem.~6.15]{LoughranSantensSurvey} (or \cite[Lem.~5.26]{LoughranSantens2024}).
    By \cref{lem:residue_computation}, this occurs if and only if $b(\lambda x,y) = b(x, \lambda y)$ for all $x,y, \lambda \in \F_{2^n}$.
    In particular, $b(x,y) = b(1, xy) = b(1, Q_{\Heis_{2^n}}(x,y))$ by \cref{lem:quadratic_map_Heis}.
    The map $\F_{2^n} \to \F_2, x \mapsto b(1,x)$ is linear, so the corresponding quadratic map is contained in the kernel of the map $\Psi_{\Heis_{2^n}}$, i.e. the corresponding Brauer element is trivial.
\end{proof}

\subsection{Thin subsets}

We now argue why the predictions of \cite{LoughranSantens2024,LoughranSantensSurvey} for Malle's conjecture for the permutation group $\Heis_4$ do not require us to remove a thin subset of field extensions in \cref{conj:explicit Heis_4 conjecture}.
It is expected that one has to remove the thin subset of field extensions arising from surjective \emph{breaking cocycles}; see \cite[Rem.~7.4]{LoughranSantensSurvey} and \cite[Lem.~3.38]{LoughranSantens2024}.

Let us recall the definition \cite[Def.~3.36]{LoughranSantens2024}, specialized to our setting.

\begin{definition}
	Let $G$ be a finite group and $C \subset G(-1)$ a Galois and conjugacy invariant subset. A homomorphism (i.e.\ cocycle) $\varphi: \Gamma_\Q \to G$ is $C$-\emph{breaking} if there exists a Galois-and-conjugacy orbit $c \subset C$ such that the $\Gamma_{\Q}$-action on $c \subset G(-1)_{\varphi}$ is not transitive. Here $G_{\varphi}(-1) = G(-1)_{\varphi}$ is the \textit{$\varphi$-inner twist} of $G(-1)$ \cite[Def.\ 2.8]{LoughranSantens2024}.
\end{definition}

Note that for $\Heis_4$, and more generally any regular permutation group $G$, the minimal index conjugacy classes are of order $p$, where $p$ is the smallest prime dividing $|G|$.
The following proposition therefore implies that, when ordering Galois $G$-number fields by discriminant, one should not have to remove a thin subset of fields in order to obtain the predicted leading constant in Malle's conjecture (see \cite[Lem.~3.38]{LoughranSantens2024}).
\begin{proposition}
    \label{prop:thin sets do not contribute}
	Let $G$ be a finite group and $C \subset G$ the subset of order $p$ elements, where $p$ is the smallest prime dividing $|G|$.
    Then all $C$-breaking cocycles are non-surjective.
\end{proposition}

\begin{proof}	
	We may assume without loss of generality that $C$ consists of a single Galois-and-conjugacy orbit. Let $\varphi: \Gamma_{\Q} \to G$ be a homomorphism.
	
	Let us make the Galois action $\Gamma_{\Q} \actson G_{\varphi}(-1)$ more concrete.
    Recall \cite[Lem.~4.4]{LoughranSantensSurvey} that after choosing a primitive root of unity, we may identify $G_{\varphi}(-1)$ with the Galois set whose underlying set is $G$ and whose Galois action is given by $(\sigma, g) \mapsto \varphi(\sigma) g^{\chi_{\mathrm{cycl}}(\sigma)^{-1}} \varphi(\sigma)^{-1}$, where $\chi_{\rm cycl} : \Gamma_{\Q} \to \widehat{\Z}^{\times}$ is the cyclotomic character.

    Notice that the above Galois action on $C$ factors through the map
    \[
    \Gamma_{\Q} \xrightarrow{(\varphi,\chi_{\rm cycl})} G \times (\Z/p\Z)^{\times}.
    \]
    We now suppose that $\varphi$ is surjective.
	Since both $\varphi$ and $\chi_{\rm cycl}$ are surjective onto their respective factors of $G \times (\Z/p\Z)^{\times}$, we deduce that the map $(\varphi, \chi_{\mathrm{cycl}})$ is surjective because $|G|$ and $|(\Z/p \Z)^{\times}| = p - 1$ are coprime by the minimality of $p$.
    This implies that the Galois action on $C \subset G_{\varphi}(-1)$ is transitive, therefore $\varphi$ is not $C$-breaking.
\end{proof}

\subsection{Brauer transforms}

\begin{lemma}
\label{lem:twisted_mass_series}
    Let $n \in \Z_{\ge 1}$.
    For every prime $p>2$ and Brauer class $\beta \in \Br_C^e B\Heis_{2^n}$ corresponding to an $\F_2$-bilinear form $b : \F_{2^n} \times \F_{2^n} \to \F_{2}$, where $C \subset \Heis_{2^n}$ is the subset of order $2$ elements, one has
    \begin{align*}
    \frac{1}{|\Heis_{2^n}|} \sum_{\varphi \in \Hom(\Gamma_{\Q_p},\Heis_{2^n})}
    \frac{e^{2\pi i \langle \beta,\varphi\rangle_p}}{|\disc(\varphi)|^{s}} &=
    1 + 3(2^n-1) p^{-2^{3n-1}s} + N_{b} p^{-3\cdot 2^{3n-2}s},
    \end{align*}
    where
    \[
    N_{b} := \#\{
    (x,y) \in \F_{2^n}^\times \times \F_{2^n}^\times :
    \forall \lambda \in \F_{2^n}, \ b(\lambda x,y) + b(x,\lambda y) = 0
    \}.
    \]
\end{lemma}

\begin{proof}
	This immediately follows from the Brauer twisted mass formula \cite[Thm.~8.23]{LoughranSantens2024} and the computation of the residues in  \cref{lem:residue_computation}.
\end{proof}

\begin{corollary}
    \label{cor:odd prime local Brauer transforms}
	Let $n \in \Z_{\ge 1}$.
    For every prime $p>2$ and Brauer class $\beta \in \Br_C^e B\Heis_{2^n}$ corresponding to an $\F_2$-bilinear form $b : \F_{2^n} \times \F_{2^n} \to \F_{2}$, where $C \subset \Heis_{2^n}$ is the subset of order $2$ elements, one has
	\[
	\widehat{\tau}_{\disc,p}(B\Heis_{2^n};\beta) = 1 + 3(2^n-1) p^{-1} + N_{b} p^{-3/2}.
	\]
\end{corollary}

\begin{proof}
	This is \cref{lem:twisted_mass_series} for $s = 2^{1-3n}$.
\end{proof}

\begin{lemma}
    \label{lem:N_beta computation for Heis_4}
    Let $C \subset \Heis_4$ be the subset of order $2$ elements.
    For every Brauer class $\beta \in \Br_C^{e} B\Heis_4$ corresponding to an $\F_2$-bilinear form $b : \F_4 \times \F_4 \to \F_2$, one has
    \[
    N_{b} = \begin{cases}
        9 & \beta \textnormal{ is trivial}, \\
        3 & \textnormal{otherwise}.
    \end{cases}
    \]
\end{lemma}

\begin{proof}
    Write $\F_4 = \F_2[\omega]$, where $\omega^2 = \omega+1$.
    Notice by the $\F_2$-bilinearity of $b$ that
    \[
    \{(x,y) \in \F_4^{\times} \times \F_{4}^{\times} : \forall \lambda \in \F_4^{\times}, b(\lambda x,y) + b(x,\lambda y) = 0\}
    =
    \{(x,y) \in \F_4^{\times} \times \F_{4}^{\times} : b(\omega x,y) + b(x,\omega y) = 0\}.
    \]
    Also, the above subset of $\F_4 \times \F_4$ is preserved under the map $(x,y) \mapsto (\omega x,\omega y)$, therefore $N_{b}$ is a multiple of $3$, hence $N_{b} \in \{0,3,6,9\}$.
    For every $x \in \F_4^{\times}$, $b(\omega x,-)+b(x,\omega-):\F_4 \to \F_2$ is a linear form, hence the contribution to $N_{b}$ from a fixed $x \in \F_4^{\times}$ is either $2^{1} - 1 = 1$ or $2^{2} - 1 = 3$, therefore summing over all $x \in \F_4^{\times}$ implies that $N_{b} \in \{3,5,7,9\}$.
    Therefore, we deduce that $N_{b} \in \{3,9\}$.

    We have that $N_{b} = 9$ if and only if for all $x,y \in \F_4$ one has that $b(x,y) = b(xy,1)$, which implies that $b(-,-)$ is given by the pullback of the linear form $b(-,1)$ along the multiplication map $\F_4 \times \F_4 \to \F_4$.
    Conversely, if $\beta$ is trivial, then we may take $b = 0$, in which case $N_0 = 9$.
\end{proof}

\section{Computing the masses at $2$ and $\infty$}
\label{sec:first section on 2-adic mass}

We first handle the $\infty$-adic Brauer transforms of $\Heis_{2^n}$ by abstract considerations.

\begin{lemma}
    \label{lem:archimedean local Brauer transforms for general groups}
    Let $G$ be a finite regular permutation group and $C \subset G$ the subset of order $2$ elements.
    Assume that $C$ generates $G$.
    Then for every $\beta \in \Br_{C}^{e} BG$, one has
    \[
    \widehat{\tau}_{\disc,\infty}(BG;\beta) = \frac{|G[2]|}{|G|}.
    \]
\end{lemma}

\begin{proof}
    The case of $\beta = 0$ is immediate; see \cite[\S7.1.1]{LoughranSantensSurvey}.
    For arbitrary $\beta \in \Br_C^e BG$, we reduce to the previous case by observing that the embedding problem
    \[
    \begin{tikzcd}
                &                &                     & \Gamma_{\R} \arrow[d] \arrow[ld, "?"', dotted] &   \\
    0 \arrow[r] & \F_2 \arrow[r] & G_{\beta} \arrow[r] & G \arrow[r]                                    & 1
    \end{tikzcd}
    \]
    is always solvable by \cref{lem:Brauer_group_discriminant}.
\end{proof}

\begin{corollary}
    \label{cor:archimedean local Brauer transforms}
    Let $n \in \Z_{\ge 1}$ and $\beta \in \Br_C^e B\Heis_{2^n}$, where $C \subset \Heis_{2^n}$ is the subset of order $2$ elements.
    Then
    \[
    \widehat{\tau}_{\disc,\infty}(B\Heis_{2^n};\beta) = \frac{|\Heis_{2^n}[2]|}{|\Heis_{2^n}|} = \frac{2\cdot 2^{2n}- 2^n}{2^{3n}}
    = 2^{1-n} - 2^{-2n}.
    \]
\end{corollary}

It remains to compute the $2$-adic mass of $\Heis_4$, which is more involved.
For any finite group $G$, the $2$-adic mass $\tau_{\disc,2}(BG)$ is effectively computable using Krasner's lemma, although a brute-force search through all \'etale algebras of degree at most $|G|$ becomes computationally unreasonable once $|G|$ is moderately large.
To alleviate this problem for the family of Heisenberg groups, we prove the following theorem that allows us to compute the $2$-adic mass more efficiently.

\begin{definition}
\label{def:admissible triple}
Let $G$ be a finite group.
For every triple of elements $(g_{-1},g_{5},g_{2})$ in $G$, define
\[
G_1:=\langle g_{-1},g_2,[g_{-1},g_5]\rangle,\quad
G_5:=\langle g_2,[g_5,g_2],[g_{-1},g_2]\rangle,
\]
\[
G_{13}:=\langle [g_{-1},g_2],g_2^2\rangle,\quad
G_{29}:=\langle g_2^2\rangle,
\]
as well as
\[
d(g_{-1},g_{5},g_{2})
:=
4\left(1-\frac1{|G_1|}\right)
+2\left(1-\frac1{|G_5|}\right)
+\left(1-\frac1{|G_{13}|}\right)
+\left(1-\frac1{|G_{29}|}\right).
\]
We say that a triple of elements $(g_{-1},g_{5},g_{2})$ in $G$ is an \textit{admissible triple} of $G$ if $g_{-1}^2 = [g_5,g_2]$.
\end{definition}

\begin{theorem}
    \label{thm:bijection between 2-adic homomorphisms and admissible triples}
    Let $G$ be a finite $2$-group of nilpotency class $\le 2$ and exponent $\le 4$.
    Then the following statements hold.
    \begin{enumerate}
    \item[(i)] There is a bijection between continuous homomorphisms $\varphi : \Gamma_{\Q_2} \to G$ and admissible triples.
    \item[(ii)] If $\varphi:\Gamma_{\Q_2} \to G$ corresponds to an admissible triple $(g_{-1},g_{5},g_{2})$, then the Galois discriminant exponent of $\varphi$ is equal to $\frac{|G|}{2} \cdot d(g_{-1},g_{5},g_{2})$.
    \item[(iii)] If $\widetilde{G}$ is another finite $2$-group of nilpotency class $\le 2$ and exponent $\le 4$ and $\pi : \widetilde{G} \to G$ is a homomorphism, then $\varphi:\Gamma_{\Q_2} \to G$ factors through $\pi$ if and only if the admissible triple $(g_{-1},g_{5},g_{2})$ of $G$ corresponding to $\varphi$ lifts to an admissible triple $(\widetilde{g}_{-1},\widetilde{g}_5,\widetilde{g}_2)$ of $\widetilde{G}$ along $\pi$.
    \end{enumerate}
\end{theorem}

\begin{corollary}
    \label{cor:2-adic mass in terms of admissible triples}
    Let $n \in \Z_{\ge 1}$ and $G = \Heis_{2^n}$.
    Then for every Brauer class $\beta \in \Br_C^e B\Heis_{2^n}$ corresponding to an $\F_2$-bilinear form $b : \F_{2^n} \times \F_{2^n} \to \F_2$, where $C \subset G$ is the subset of order $2$ elements, one has
    \[
    \widehat{\tau}_{\disc,2}(BG;\beta) = 
    \frac{1}{|G|} \sum_{(g_{-1},g_{5},g_{2})} 2^{-d(g_{-1},g_{5},g_{2})}
    (-1)^{b(x_{-1},y_{-1}) + b(x_{5},y_{2}) + b(x_{2},y_{5})},
    \]
    where $(g_{-1},g_5,g_2) \in G \times G \times G$ varies over all admissible triples and $g_i = (x_i,y_i,z_i)$ for $i \in \{-1,5,2\}$.
\end{corollary}

\begin{proof}
    The case of the trivial class $\beta = 0$ is immediate by \cref{thm:bijection between 2-adic homomorphisms and admissible triples}(i) and (ii).
    For arbitrary $\beta$, the sign of $\pm 2^{-d(g_{-1},g_{5},g_{2})}$ appearing in $\widehat{\tau}_{\disc,2}(BG;\beta)$ is determined by the solvability of the embedding problem
    \[
    \begin{tikzcd}
                &                &                     & \Gamma_{\Q_2} \arrow[d] \arrow[ld, "?"', dotted] &   \\
    0 \arrow[r] & \F_2 \arrow[r] & G_{\beta} \arrow[r] & G \arrow[r]                                      & 1.
    \end{tikzcd}
    \]
    In other words, the sign of $\pm 2^{-d(g_{-1},g_{5},g_{2})}$ is $+1$ if the embedding problem is solvable, and $-1$ otherwise.

    Both $G$ and $G_{\beta}$ are central extensions of a finite-dimensional $\F_2$-vector space by another such vector space, therefore $G$ and $G_{\beta}$ are both finite $2$-groups of nilpotency class $\le 2$ and exponent $\le 4$, so we may apply \cref{thm:bijection between 2-adic homomorphisms and admissible triples}(iii) to the map $\pi : G_{\beta} \to G$.
    
    For all $i \in \{-1,5,2\}$, write $g_i = (x_i,y_i,z_i) \in G$.
    A triple $(g_{-1},g_{5},g_{2})$ of $G$ is admissible if and only if
    \[
    Q_G(x_{-1},y_{-1}) = \langle (x_5,y_5),(x_2,y_2)\rangle_{G}.
    \]
    If $(\widetilde{g}_{-1},\widetilde{g}_{5},\widetilde{g}_2)$ is an arbitrary lift of the triple $(g_{-1},g_{5},g_{2})$ to $G_{\beta}$, then in coordinates we have for all $i \in \{-1,5,2\}$ that $\widetilde{g}_i = (x_i,y_i,z_i \oplus t_i)$, where $t_i \in \F_2$.
    Since $Q_{G_{\beta}} = (Q_G,Q_b)$, we have that $(\widetilde{g}_{-1},\widetilde{g}_{5},\widetilde{g}_2)$ is admissible if and only if 
    \[
    Q_G(x_{-1},y_{-1}) \oplus Q_b(x_{-1},y_{-1}) = \langle (x_5,y_5),(x_2,y_2)\rangle_G \oplus \langle (x_5,y_5),(x_2,y_2)\rangle_b.
    \]
    In other words, $(\widetilde{g}_{-1},\widetilde{g}_{5},\widetilde{g}_2)$ is admissible if and only if $(g_{-1},g_5,g_2)$ is admissible and the quantity
    \[
    Q_b(x_{-1},y_{-1}) + \langle(x_5,y_5),(x_2,y_2)\rangle_b = 
    b(x_{-1},y_{-1}) + b(x_5,y_2) + b(x_2,y_5)
    \]
    vanishes.
    Thus, the correct sign for the expression $\pm 2^{-d(g_{-1},g_{5},g_{2})}$ appearing in the expression for $\widehat{\tau}_{\disc,2}(BG;\beta)$ is given by
    \[
    (-1)^{ b(x_{-1},y_{-1}) + b(x_5,y_2) + b(x_2,y_5)}.
    \]
    This finishes the proof.
\end{proof}

\begin{corollary}
\label{cor:numerical 2-adic mass of Heis_4}
Let $G = \Heis_4$, $C \subset G$ be the subset of order $2$ elements, and $\beta \in \Br_C^e B\Heis_4$ correspond to an $\F_2$-bilinear form $b:\F_4 \times \F_4 \to \F_2$.
Then one has
\[
\widehat{\tau}_{\disc,2}(BG;\beta) = 
\frac{1}{16}
\begin{cases}
    313+54\sqrt{2}+54\sqrt[4]{2}+9\sqrt{2}\sqrt[4]{2} & \beta \textnormal{ is trivial}, \\
    247 + 30\sqrt[4]{2} + 36\sqrt{2} - 3\sqrt{2}\sqrt[4]{2} & \textnormal{otherwise}.
\end{cases}
\]
\end{corollary}

\begin{proof}
    See the verification code in the accompanying \href{https://github.com/JackMillerMath/2-adic-mass-computation-for-Heis_4-regular-representation}{\textnormal{GitHub repository}}.\footnote{\url{https://github.com/JackMillerMath/2-adic-mass-computation-for-Heis_4-regular-representation}}
\end{proof}

\section{Proof of \cref{thm:intro thm presentation and ramification filtration at 2}}
\label{sec:second section on 2-adic mass}

In this section, we prove \cref{thm:intro thm presentation and ramification filtration at 2,thm:bijection between 2-adic homomorphisms and admissible triples}.

\subsection{Nilpotent $2$-groups of bounded class and exponent}

For all $c,m \in \Z_{\ge 0}$ we define $\mc{N}_{c,m}$ to be the set of isomorphism classes of finite $2$-groups that have nilpotency class $\le c$ and exponent $\le 2^m$.
We note that $\mc{N}_{c,m}$ is a \textit{formation}, meaning that $\mc{N}_{c,m}$ is closed under taking quotients and finite subdirect products in the category of finite groups.
For every $2$-adic local field $k$, we define $\mc{G}_{c,m}(k)$ to be the pro-$\mc{N}_{c,m}$-completion of the profinite group $\Gamma_{k}$.

\begin{lemma}
    \label{lem:finitely generated formation}
    For all $c,m \in \Z_{\ge0}$, $\mc{N}_{c,m}$ is a finitely generated formation.
\end{lemma}

\begin{proof}
    It is a well-known statement in universal algebra \cite[Cor.\ 35.12]{NeumannVarietiesOfGroups} that the variety $\mc{V}$ of nilpotent groups of class at most $c$ is generated by the groups $H \in \mc{V}$ such that $H$ is generated by at most $c$ elements.
    In particular, if we consider the variety of groups
    \[
    \mc{V}_{c,m} := {\rm Var}\left(
    x^{2^m} = 1, [x_1,[x_2,[\ldots,[x_c,x_{c+1}]]]] = 1
    \right),
    \]
    then we observe that $\mc{V}_{c,m}$ is generated by the universal group of class $\le c$, exponent dividing $2^m$, and generated by at most $c$ elements given by
    \[
    F_c(\mc{V}_{c,m}) := F_c / \langle y^{2^m}, \ [y_1,[y_2,[\ldots,[y_c,y_{c+1}]]]] : y,y_1,\ldots,y_{c+1} \in F_c \rangle,
    \]
    where $F_c$ is the free group on $c$ letters.

    It remains to show that $F_c(\mc{V}_{c,m})$ is a finite group; then $F_c(\mc{V}_{c,m}) \in \mc{N}_{c,m}$ and generates $\mc{N}_{c,m}$ as a formation.
    Indeed, we observe that the lower central series quotients of $F_c(\mc{V}_{c,m})$ are abelian groups of exponent dividing $2^m$, and the $i$-th such quotient is generated by the weight $i$ commutators of the basis of $F_c$,  which for every $i \in \Z_{\ge 1}$ is a finite generating set.
    Since finitely generated torsion abelian groups are finite, this proves that $F_c(\mc{V}_{c,m})$ admits a $c$-step filtration all of whose quotients are finite, therefore $F_c(\mc{V}_{c,m})$ is finite.
\end{proof}

\begin{lemma}
    If $\mc{L}$ is a finitely generated formation (also called a \emph{level}) of finite groups, and $X$ is a small profinite group, then the pro-$\mc{L}$-completion $X^{\mc{L}}$ is finite.
\end{lemma}

\begin{proof}
    This immediately follows from \cite[Lem.\ 5.7]{SawinWoodMomentMethod} applied to the diamond category of finite groups \cite[Lem.\ 6.10]{SawinWoodMomentMethod}.
\end{proof}

\begin{corollary}
    \label{cor:maximal nilpotent quotients are finite}
    For every $2$-adic local field $k$ and $c,m \in \Z_{\ge 0}$, the group $\mc{G}_{c,m}(k)$ is finite.
\end{corollary}

\begin{proof}
    We use \cref{lem:finitely generated formation} and the well-known fact that the absolute Galois group of a local field is a small profinite group.
\end{proof}

\subsection{Reduction to computing $\mc{G}_{c,m}(\Q_2)$}

Let $c,m \in \Z_{\ge 0}$ and $G \in \mc{N}_{c,m}$.
Then every continuous homomorphism
\[
\Gamma_{\Q_2} \to G
\]
factors through $\mc{G}_{c,m}(\Q_2)$ by definition as the pro-$\mc{N}_{c,m}$-completion of $\Gamma_{\Q_2}$.
Therefore the task of computing the $2$-adic mass of $G$ essentially reduces to computing the finite group $\mc{G}_{c,m}(\Q_2)$ (see \cref{cor:maximal nilpotent quotients are finite}) and its \textit{(lower numbering) ramification filtration}
\[
\mc{G}_{c,m}(\Q_2) \supset \mc{G}_{c,m}(\Q_2)_0 \supset \mc{G}_{c,m}(\Q_2)_1 \supset \dotsi.
\]
Indeed, by Galois correspondence, $\mc{G}_{c,m}(\Q_2)$ is canonically identified with the automorphism group of a finite Galois extension $K_{c,m} / \Q_2$.
Thus, for all $i \in \Z_{\ge -1}$ we define the \textit{$i$-th ramification group $\mc{G}_{c,m}(\Q_2)_i$} to be $\Gal(K_{c,m}/\Q_2)_i$.

\begin{lemma}
    \label{lem:2-adic mass in terms of conductor discriminant formula}
    Let $c,m\in\Z_{\ge 0}$ and $G \in \mc{N}_{c,m}$ be viewed as a regular permutation group.
    Then
    \[
    \tau_{\disc,2}(BG) = 
    \frac{1}{|G|}
    \sum_{\varphi : \mc{G}_{c,m}(\Q_2) \to G}
    2^{-d_{\rm reg}(\varphi)},
    \]
    where
    \[
    d_{\rm reg}(\varphi) := 
    2
    \sum_{i \ge 0} 
    \frac{|\mc{G}_{c,m}(\Q_2)_i|}{|\mc{G}_{c,m}(\Q_2)_0|}
     \left(
        1 - \frac{1}{|\varphi(\mc{G}_{c,m}(\Q_2)_i)|}
    \right).
    \]
\end{lemma}

\begin{proof}
    Note that $d_{\rm reg}(\varphi)$ should be defined to equal $\frac{2}{|G|}$ times the discriminant exponent of the Galois $G$-structured \'etale algebra associated to $\varphi$.
    The exact formula for $d_{\rm reg}(\varphi)$ therefore follows from the conductor-discriminant formula \cite[VII.11.9]{NeukirchANT}.
\end{proof}

\subsection{Computing $\mc{G}_{2,2}(\Q_2)$}

Since $\Heis_4 \in \mc{N}_{2,2}$, we focus on explicitly computing $\mc{G}_{2,2}(\Q_2)$ and its ramification filtration.
We do this by writing down an exact sequence for $\mc{G}_{2,2}(\Q_2)$ in terms of $\mc{G}_{1,1}$ groups, applying Kummer theory to compute these smaller groups and their ramification filtrations, and then combining the ramification filtrations to obtain the complete information about the ramification filtration on $\mc{G}_{2,2}(\Q_2)$.

\begin{notation}
    If $X$ is a profinite group and $S_1,\ldots,S_k \subset X$ are subsets with $k\ge 2$, we let $S_1 \dotsi S_k \le X$ denote the topological closure of the subgroup generated by $S_1 \cup \dotsi \cup S_k$.
    If $q \in \Z_{\ge 0}$ and $S,T \subset X$ are subsets, we let $S^q$ denote the image of the $q$-th power map $S \to X$ and $[S,T]$ denote the image of the commutator map $S \times T \to X$, $(s,t) \mapsto sts^{-1}t^{-1}$.
\end{notation}

\begin{lemma}
    \label{lem:identities in N22}
    Let $X$ be a pro-$2$-group that has nilpotency class $\le 2$ and exponent $\le 4$.
    For all $x,y \in X$, the elements $x^2,[x,y]$ have order $\le 2$ and are central.
\end{lemma}

\begin{proof}
    It follows from \cite[Prop.\ 3.8.3]{NeukirchSchmidtWingberg} that for all $x,y,z \in X$ one has
    \[
    (xy)^4 = x^4y^4[x,y]^{\binom{4}{2}}, \qquad 
    [x^2,z] = [x,z]^2.
    \]
    The first identity implies that $[x,y]^2 = 1$, hence $[x,y]$ has order $\le 2$; in other words, all commutators have order $\le 2$ (and are central because $\gamma_3(X) = 1$).
    The second identity now implies that $[x^2,z] = 1$, therefore $x^2$ is central (and order $\le 2$ because $X$ has exponent $\le 4$).
\end{proof}

\begin{definition}
    For a pro-$2$-group $X$, we define the \textit{lower $2$-central series} $(X_n)_{n=1}^{\infty}$ via
    \[
    X_1 := X, \quad X_{n+1} := X_{n}^2 [X_n,X].
    \]
    We also define the \textit{weight $n$ commutator subgroups} $(\gamma_n(X))_{n=1}^{\infty}$ via
    \[
    \gamma_1(X) := X, \quad 
    \gamma_{n+1}(X) := [\gamma_n(X),X].
    \]
\end{definition}

\begin{lemma}
    \label{lem:first few lower 2-central series subgroups}
    Let $X$ be a pro-$2$-group.
    Then one has
    \[
    X_2 = X^2 \gamma_2(X), \qquad 
    X_3 = X^4 \gamma_3(X).
    \]
\end{lemma}

\begin{proof}
    The first identity is immediate from the definition of $X_2$.
    For the second identity, we note that $X_3$ contains fourth powers and triple commutators, therefore $X_3 \supset X^4 \gamma_3(X)$.
    After quotienting by the normal subgroup $X^4 \gamma_3(X)$, we may reduce to the case that $X$ has nilpotency class $\le 2$ and exponent $\le 4$, where we would like to show that $X_3$ is trivial.

    Thus, let $X$ be a pro-$2$-group of nilpotency class $\le 2$ and exponent $\le 4$.
    Applying \cref{lem:identities in N22} to $X$ shows that $X_2$ is generated by central elements of order $\le 2$.
    In particular, it is immediately clear that $X_2^2 = [X_2,X] = 1$ and therefore $X_3 = 1$.
\end{proof}

\begin{proposition}
    \label{prop:G22 computed by G11's}
    There is a canonical central extension of finite $\Gamma_{\Q_2}$-groups
    \[
    1 \to \mc{G}_{1,1}(K_{1,1})_{\Gamma_{\Q_2}} \to \mc{G}_{2,2}(\Q_2) \to \mc{G}_{1,1}(\Q_2) \to 1
    \]
    where both $\mc{G}_{1,1}(K_{1,1})$ and $\mc{G}_{1,1}(\Q_2)$ are finite-dimensional $\F_2$-vector spaces.
    Moreover, $\mc{G}_{1,1}(K_{1,1})_{\Gamma_{\Q_2}} \le \mc{G}_{2,2}(\Q_2)$ is equal to the subgroup generated by the squares and commutators of lifts of elements of $\mc{G}_{1,1}(\Q_2)$ to $\mc{G}_{2,2}(\Q_2)$.
\end{proposition}

\begin{proof}
    The finiteness statements follow from \cref{cor:maximal nilpotent quotients are finite}.

    We let $X$ denote the pro-$2$-completion of $\Gamma_{\Q_2}$.
    Then \cref{lem:first few lower 2-central series subgroups} implies that there are canonical isomorphisms
    \[
    \mc{G}_{1,1}(\Q_2) \cong X/X_2, \qquad 
    \mc{G}_{2,2}(\Q_2) \cong X/X_3, \qquad 
    \ker(\mc{G}_{2,2}(\Q_2) \twoheadrightarrow \mc{G}_{1,1}(\Q_2)) \cong X_2/X_3.
    \]
    We write
    \[
    X_2/X_3 = 
    \frac{X_2}{X_2^2[X_2,X]}.
    \]
    
    By definition of the lower $2$-central series $(X_n)_{n=1}^{\infty}$, it is clear that $X/X_3$ is the maximal central extension of $X/X_2$ by a group that is abelian and exponent $\le 2$.
    On the other hand, Galois correspondence says that $\mc{G}_{1,1}(K_{1,1})$ corresponds to the maximal extension of $K_{1,1}$ that is abelian and exponent $\le 2$, in other words
    \[
    \mc{G}_{1,1}(K_{1,1}) \cong \frac{X_2}{X_2^2 [X_2,X_2]}.
    \]
    Since $\mc{G}_{1,1}(K_{1,1})$ as a $\Gamma_{\Q_2}$-module is inflated from its $X$-module structure, we can compute the coinvariants of $\mc{G}_{1,1}(K_{1,1})$ as
    \[
    \mc{G}_{1,1}(K_{1,1})_{\Gamma_{\Q_2}} = 
     \mc{G}_{1,1}(K_{1,1})_{X}\cong 
    \frac{X_2}{X_2^2[X_2,X_2][X_2,X]} = X_2/X_3.
    \]
    Since $X_2$ is generated by the squares and commutators of $X$, it follows that $\mc{G}_{1,1}(K_{1,1})_{\Gamma_{\Q_2}} \cong X_2/X_3$ is generated by the squares and commutators of lifts of elements of $\mc{G}_{1,1}(\Q_2) \cong X/X_2$ to $\mc{G}_{2,2}(\Q_2) \cong X/X_3$.
\end{proof}

\subsection{The ramification filtration on $\mc{G}_{2,2}(\Q_2)$}
\label{subsec:ramification filtration on mcG22}

Using Kummer theory, we write
\[
K_{1,1} = \Q_2(\sqrt{-1},\sqrt{5},\sqrt{2}), \qquad
\mc{G}_{1,1}(\Q_2) \cong \F_2^3 = \langle \bar{\sigma}_{-1}, \bar{\sigma}_{5}, \bar{\sigma}_{2}\rangle,
\]
where $\bar{\sigma}_{-1}, \bar{\sigma}_{5}, \bar{\sigma}_{2}$ form a dual basis for $-1,5,2 \in \Q_2^\times/\Q_2^{\times 2}$.
Although a choice of lifts $\sigma_{-1}, \sigma_{5}, \sigma_{2} \in \mc{G}_{2,2}(\Q_2)$ is noncanonical, the lifted squares $\sigma_{i}^2$ and commutators $[\sigma_i,\sigma_j]$, which lie in $\mc{G}_{1,1}(K_{1,1})_{\Gamma_{\Q_2}}$, are in fact canonically well-defined for all $i,j \in \{-1,5,2\}$ and span $\mc{G}_{1,1}(K_{1,1})_{\Gamma_{\Q_2}}$ as an $\F_2$-vector space.

\begin{lemma}
    \label{lem:five dimensional top coinvariant Kummer layer}
    $\mc{G}_{1,1}(K_{1,1})_{\Gamma_{\Q_2}}$ is canonically presented as a five-dimensional $\F_2$-vector space given by
    \[
    \mc{G}_{1,1}(K_{1,1})_{\Gamma_{\Q_2}} =
    \F_2\langle \forall i,j \in \{-1,5,2\}:  \sigma_i^2, [\sigma_i,\sigma_j] \mid \sigma_{-1}^2 = [\sigma_{5},\sigma_{2}]
    \rangle .
    \]
\end{lemma}

\begin{proof}
    This follows from the theory of Demushkin groups \cite[Def.\ 3.9.9]{NeukirchSchmidtWingberg}.
    We let $X$ denote the pro-$2$-completion of $\Gamma_{\Q_2}$.
    By \cite[Propositions 3.9.1 and 3.9.5]{NeukirchSchmidtWingberg} and the fact that $H^1(\Q_2,\F_2)$ and $H^2(\Q_2,\F_2)$ are $\F_2$-vector spaces of dimensions $3$ and $1$, respectively, we have a minimal presentation
    \[
    1 \to R \to F \to X \to 1
    \]
    where $F$ is a free pro-$2$-group of rank $3$ mapping its generators $x_{-1},x_{5},x_{2}$ to some choice of lifts $\sigma_{-1}, \sigma_{5}, \sigma_{2} \in X$, and $R$ is a free pro-$2$-group of rank $1$.
    Applying \cite[Prop.\ 3.9.13(ii)]{NeukirchSchmidtWingberg} shows that $R$ is generated by elements of $F$ of the form
    \[
    x_{-1}^{2 a} x_{5}^{2b} x_{2}^{2c} [x_{-1},x_{5}]^{d} [x_{-1},x_{2}]^{e} [x_{5},x_{2}]^{f} \text{ modulo } F_3,
    \]
    where the powers $a,b,c,d,e,f \in \F_2$ are determined by the cup product on $H^1(\Q_2,\F_2)$, or equivalently the table of Hilbert symbols:
    \[
    \begin{array}{c|ccc}
    (\cdot,\cdot)_2 & -1 & 5 & 2 \\
    \hline
    -1 & -1 & 1 & 1 \\
    5  & 1  & 1 & -1 \\
    2  & 1  & -1 & 1
    \end{array}.
    \]
    The table demonstrates that $a = f = 1$ and all other powers are zero, yielding the relation $\sigma_{-1}^2 [\sigma_5,\sigma_2] = 1 \mod X_3$, or equivalently $\sigma_{-1}^2 = [\sigma_5,\sigma_2] \mod X_3$ because $X_2/X_3$ is an $\F_2$-vector space.
    This concludes the proof because \cref{lem:first few lower 2-central series subgroups,prop:G22 computed by G11's} imply that $\mc{G}_{1,1}(K_{1,1})_{\Gamma_{\Q_2}} \cong X_2/X_3$.
\end{proof}

\begin{proposition}[Ramification filtration]
\label{prop:ramification filtration}
Let $\sigma_{-1}, \sigma_{5}, \sigma_{2} \in \mc{G}_{2,2}(\Q_2)$ be arbitrary lifts of the Kummer dual basis $\bar{\sigma}_{-1}, \bar{\sigma}_{5}, \bar{\sigma}_2 \in \mc{G}_{1,1}(\Q_2)\cong (\Q_2^{\times}/\Q_2^{\times 2})^\vee$.
Then the ramification filtration on $\mc{G}_{2,2}(\Q_2)$ has lower numbering breaks at $-1,1,5,13,29$, and the corresponding ramification subgroups are
\[
\mc{G}_{2,2}(\Q_2)_i =
\begin{cases}
    \langle \sigma_{-1},\sigma_{5}, \sigma_2\rangle & i = -1, \\
    \langle \sigma_{-1}, [\sigma_{-1},\sigma_{5}], \sigma_{2}\rangle & -1 < i \le 1, \\
    \langle [\sigma_5,\sigma_2], [\sigma_{-1},\sigma_{2}], \sigma_{2}\rangle & 1 < i \le 5, \\
    \langle [\sigma_{-1},\sigma_{2}], \sigma_2^2 \rangle & 5 < i \le 13, \\
    \langle \sigma_{2}^2\rangle & 13 < i \le 29. \\
\end{cases}
\]
\end{proposition}

\begin{table}
    \centering
    \begin{tabular}{|c|c|c|c|c|c|c|}
        \hline
        $\chi \in W$ & $\rho \in \mc{G}_{1,1}(K_{1,1})_{\Gamma_{\Q_2}}$ & $F_{\chi}/\Q_2$ & $E_{\chi}/K_{1,1}$ & $E_{\chi}'/F_{\chi}$ & $\Gal(E_{\chi}'/\Q_2)$ & Exp. \\
        \hline
        $\chi_{5,5}$ & $\sigma_{5}^2$ & \href{https://www.lmfdb.org/padicField/2.2.1.0a1.1}{$\Q_2(\sqrt{5})$} & \href{https://www.lmfdb.org/padicField/2.4.4.32b1.1}{\texttt{2.4.4.32b1.1}} & \href{https://www.lmfdb.org/padicField/2.4.1.0a1.1}{\texttt{2.4.1.0a1.1}} & $C_4$ &  0 \\
        $\chi_{-1,5}$ & $[\sigma_{-1},\sigma_{5}]$ & \href{https://www.lmfdb.org/padicField/2.2.2.4a1.1}{$\Q_2(\sqrt{-1},\sqrt{5})$} & \href{https://www.lmfdb.org/padicField/2.2.8.36b1.1}{\texttt{2.2.8.36b1.1}} & \href{https://www.lmfdb.org/padicField/2.2.4.12a1.1}{\texttt{2.2.4.12a1.1}} & $D_4$ &  2 \\
        $\chi_{5,2}$ & $[\sigma_{5},\sigma_{2}]$ & --- & --- & --- & --- &  (4) \\
        $\chi_{-1,2}$ & $[\sigma_{-1},\sigma_{2}]$ & \href{https://www.lmfdb.org/padicField/2.1.4.8b1.1}{$\Q_2(\sqrt{-1},\sqrt{2})$} & \href{https://www.lmfdb.org/padicField/2.2.8.44d1.44}{\texttt{2.2.8.44d1.44}} & \href{https://www.lmfdb.org/padicField/2.1.8.22d1.9}{\texttt{2.1.8.22d1.9}} & $D_4$ &  6 \\
        $\chi_{2,2}$ & $\sigma_{2}^2$ & \href{https://www.lmfdb.org/padicField/2.1.2.3a1.3}{$\Q_2(\sqrt{2})$} & \href{https://www.lmfdb.org/padicField/2.2.8.48c1.1}{\texttt{2.2.8.48c1.1}} & \href{https://www.lmfdb.org/padicField/2.1.4.11a1.11}{\texttt{2.1.4.11a1.11}} & $C_4$ &  8 \\
        \hline
    \end{tabular}
    \caption{Given a basis element $\chi \in W$ with dual element $\rho \in \mc{G}_{1,1}(K_{1,1})_{\Gamma_{\Q_2}}$, we record a multiquadratic field $F_{\chi}/\Q_2$ and a quadratic extension $E_{\chi}/K_{1,1}$ that descends to a quadratic extension $E_{\chi}'/F_{\chi}$.
    The fourth and fifth columns use \cite{LMFDB} \texttt{p-adic field} labels.
    The last column records the \cite{LMFDB} relative discriminant exponent of $E_{\chi}/K_{1,1}$, which is equal to the conductor exponent of $\chi$.
    The deduced conductor exponent for $\chi_{5,2}$ is parenthetically displayed in the last column.}
    \label{tab:LMFDB fields associated to characters}
\end{table}

\begin{proof}
    \cite{LMFDB} contains a complete list of all $2$-adic number fields of degree $\le 16$, which in the range considered here lists the relevant fields together with their discriminants, Galois groups, and subfield data.
    Notice that $K_{1,1} = \Q_2(\sqrt{-1},\sqrt{5},\sqrt{2})$ corresponds to \cite[\href{https://www.lmfdb.org/padicField/2.2.4.16b1.1}{\texttt{2.2.4.16b1.1}}]{LMFDB}.

    The API request \cite[\href{https://www.lmfdb.org/api/lf_fields/?p=i2&n=i16&galois_degree=i16&subfield=cs2.8.16.6&_fields=new_label,c&_sort=c}{query link}]{LMFDB} confirms that there are $2^5 - 1 = 31$ degree $16$ Galois $2$-adic number fields containing $K_{1,1}$, and the set of absolute discriminant exponents of these fields is $\{32,36,40,44,48\}$.
    In particular, the tower property for discriminants implies that the set of relative discriminant exponents with respect to $K_{1,1}$ is $\{0,2,4,6,8\}$.
    By \cref{lem:five dimensional top coinvariant Kummer layer}, such fields correspond to the $31$ nontrivial quadratic characters of $K_{1,1}$.
    We let $W$ denote the $\F_2$-vector space of quadratic characters of $K_{1,1}$.
    By local class field theory, the set of conductor exponents of $\chi \in W\backslash\{0\}$ is $\{0,2,4,6,8\}$.

    Local class field theory implies that the upper numbering ramification groups $\mc{G}_{1,1}(K_{1,1})_{\Gamma_{\Q_2}}^\bullet = (\mc{G}_{1,1}(K_{1,1})_{\Gamma_{\Q_2}}^\bullet)_{u \ge -1}$ under the Kummer pairing $\mc{G}_{1,1}(K_{1,1})_{\Gamma_{\Q_2}} \times W \to \F_2$ satisfy the orthogonality relations
    \[
    \mc{G}_{1,1}(K_{1,1})_{\Gamma_{\Q_2}}^u = (W^u)^{\perp}, \quad 
    W^u := \{\chi \in W\backslash\{0\} : \chi \text{ has conductor exponent} < u+1\} \cup \{0\}.
    \]
    In particular, $\mc{G}_{1,1}(K_{1,1})_{\Gamma_{\Q_2}}^\bullet$ has upper numbering ramification breaks at $-1,1,3,5,7$.
    Therefore, the filtration $0 = W^{-1} < W^{1} < W^{3} < W^{5} < W^{7} < W$ is a complete flag $W^\bullet$.
    
    Our goal is to compute the flag $W^\bullet$ explicitly in terms of the basis $\chi_{5,5}, \chi_{2,2},\chi_{-1,5},\chi_{-1,2},\chi_{5,2}$, which is Kummer dual to the basis $\sigma_5^2, \sigma_2^2, [\sigma_{-1},\sigma_5], [\sigma_{-1},\sigma_2],[\sigma_5,\sigma_2]$ of $\mc{G}_{1,1}(K_{1,1})_{\Gamma_{\Q_2}}$.
    \cref{tab:LMFDB fields associated to characters} provides a method for computing the conductor exponents of this basis, which we describe as follows.
    For each basis element $\chi \in W$ in \cref{tab:LMFDB fields associated to characters}, let $\rho \in \mc{G}_{1,1}(K_{1,1})_{\Gamma_{\Q_2}}$ be the dual element.
    By construction, $\rho$ is the image of either a lifted square (resp.\ commutator) of an element (resp.\ elements) $\bar{\sigma}_{\rho}$ among $\bar{\sigma}_{-1},\bar{\sigma}_{5},\bar{\sigma}_{2}$.
    After choosing the basis $-1,5,2$ for $\Q_2^{\times}/\Q_2^{\times 2}$, Kummer theory associates to this collection of elements $\bar{\sigma}_{\rho}$ a multiquadratic field $F_{\chi}/\Q_2$ whose Galois group is generated by the images of the elements $\bar{\sigma}_{\rho}$.
    We then use \cite{LMFDB} to find a quadratic extension $E_{\chi}/K_{1,1}$ that descends to a quadratic extension $E_{\chi}'/F_{\chi}$ whose Galois group fits in an extension
    \[
    0 \to \F_2 \to \Gal(E_{\chi}'/\Q_2) \to \Gal(F_{\chi}/\Q_2) \to 1.
    \]
    We observe from \cref{tab:LMFDB fields associated to characters} that this extension is either $0 \to \F_2 \to C_4 \to \F_2 \to 0$ or $0 \to \F_{2} \to D_4 \to \F_2^{\oplus 2} \to 0$; in either case, we find that $\Gal(E_{\chi}'/F_{\chi})$ is generated by a lifted square (resp.\ commutator) of the image of the element (resp.\ elements) $\bar{\sigma}_{\rho}$, i.e.\ $\Gal(E_{\chi}'/F_{\chi})$ is generated by the image of $\rho$.
    This proves that $E_{\chi}/K_{1,1}$ corresponds to $\chi$ via Kummer theory, and in particular that the relative discriminant exponents in \cref{tab:LMFDB fields associated to characters} correspond to the conductor exponents of $\chi$.

    Thus, we have shown that $\chi_{5,5},\chi_{-1,5},\chi_{-1,2},\chi_{2,2}$ have conductor exponents $0,2,6,8$ respectively.
    We deduce that $\chi_{5,2}$ has conductor exponent $4$ because $W^\bullet$ is a complete flag.
    This implies that $\mc{G}_{1,1}(K_{1,1})_{\Gamma_{\Q_2}}$ has upper numbering ramification breaks at $-1,1,3,5,7$ given by consecutively removing the elements $\sigma_{5}^2,[\sigma_{-1},\sigma_5],[\sigma_5,\sigma_2],[\sigma_{-1},\sigma_2],\sigma_2^2$ from the basis of $\mc{G}_{1,1}(K_{1,1})_{\Gamma_{\Q_2}}$.
    To convert to the lower numbering ramification breaks on $\mc{G}_{1,1}(K_{1,1})_{\Gamma_{\Q_2}}$, it is a standard calculation using the Herbrand function that the lower numbering ramification breaks are $-1,1,5,13,29$.
    
    To obtain the ramification filtration on $\mc{G}_{2,2}(\Q_2)$, we use compatibility of the lower-numbering filtration with subgroups and of the upper-numbering filtration with quotient groups \cite{NeukirchANT}.
    Namely, for all $x \in \R_{\ge -1}$ one has a short exact sequence of groups
    \[
    0 \to (\mc{G}_{1,1}(K_{1,1})_{\Gamma_{\Q_2}})_x \to 
    \mc{G}_{2,2}(\Q_2)_x \to 
    \mc{G}_{1,1}(\Q_2)^{H(x)} \to 0
    \]
    where $H(x)$ is an explicitly computable piecewise linear function depending only on the ramification filtrations of $\mc{G}_{1,1}(K_{1,1})_{\Gamma_{\Q_2}}$ and $\mc{G}_{1,1}(\Q_2)$.

    The upper numbering ramification breaks of $\mc{G}_{1,1}(\Q_2)$ are $-1,1,2$; indeed, the fields $\Q_2(\sqrt{5})$, $\Q_2(\sqrt{-1})$, $\Q_2(\sqrt{2})$ have discriminant exponents $-1+1,1+1,2+1$ respectively.
    Having computed the ramification filtrations of both $\mc{G}_{1,1}(K_{1,1})_{\Gamma_{\Q_2}}$ and $\mc{G}_{1,1}(\Q_2)$, we can explicitly compute $H(x)$ and deduce that $\mc{G}_{2,2}(\Q_2)$ has lower numbering ramification breaks at $-1,1,5,13,29$, and that the groups
    \[
    \mc{G}_{2,2}(\Q_2) = \mc{G}_{2,2}(\Q_2)_{-1} >
    \mc{G}_{2,2}(\Q_2)_{1} > 
    \mc{G}_{2,2}(\Q_2)_{5} > 
    \mc{G}_{2,2}(\Q_2)_{13} >
    \mc{G}_{2,2}(\Q_2)_{29} > 1
    \]
    are presented by the exact sequences
    \[
    \mc{G}_{2,2}(\Q_2)_{-1} =
    \langle \sigma_{-1}, \sigma_{5}, \sigma_2\rangle,
    \]
    \[
    0 \to \langle [\sigma_{-1},\sigma_5], [\sigma_{5},\sigma_{2}], [\sigma_{-1},\sigma_{2}], \sigma_{2}^2\rangle
    \to \mc{G}_{2,2}(\Q_2)_{1} \to \langle \bar{\sigma}_{-1}, \bar{\sigma}_{2} \rangle
    \to 0,
    \]
    \[
    0 \to \langle [\sigma_{5},\sigma_{2}], [\sigma_{-1},\sigma_{2}], \sigma_{2}^2\rangle
    \to \mc{G}_{2,2}(\Q_2)_{5} \to \langle \bar{\sigma}_{2} \rangle
    \to 0,
    \]
    \[
    0 \to \langle [\sigma_{-1},\sigma_{2}], \sigma_{2}^2\rangle
    \to \mc{G}_{2,2}(\Q_2)_{13} \to 1,
    \]
    \[
    0 \to \langle\sigma_{2}^2\rangle
    \to \mc{G}_{2,2}(\Q_2)_{29} \to 1.
    \]
    From this we deduce the desired subgroup presentations, completing the proof.
\end{proof}

\subsection{Finishing the proof}

\begin{proof}[Proof of \cref{thm:intro thm presentation and ramification filtration at 2}]
    This immediately follows from \cref{subsec:ramification filtration on mcG22}.
\end{proof}

\begin{proof}[Proof of \cref{thm:bijection between 2-adic homomorphisms and admissible triples}]
    In the proof of \cref{lem:five dimensional top coinvariant Kummer layer}, we see that $\mc{G}_{2,2}(\Q_2)$ is the group of nilpotency class $2$ and exponent $4$ generated by three elements $\sigma_{-1},\sigma_5,\sigma_{2}$ subject to the single additional relation $\sigma_{-1}^2 = [\sigma_5,\sigma_2]$.
    This shows that $\Hom(\Gamma_{\Q_2},G) \cong \Hom(\mc{G}_{2,2}(\Q_2),G)$ is in bijection with admissible triples of $G$, and this bijection is compatible with lifting admissible triples along a homomorphism $\pi:\widetilde{G} \to G$, where $\widetilde{G} \in \mc{N}_{2,2}$.
    This proves \cref{thm:bijection between 2-adic homomorphisms and admissible triples}(i) and (iii).
    
    By the $c=m=2$ version of \cref{lem:2-adic mass in terms of conductor discriminant formula}, we want to show that $d_{\rm reg}(\varphi) = d(g_{-1},g_5,g_2)$ if $\varphi \in \Hom(\Gamma_{\Q_2},G)$ corresponds to $(g_{-1},g_5,g_2)$ under this bijection.
    By \cref{prop:ramification filtration}, we have that
    \begin{align*}
        d_{\rm reg}(\varphi) &=
        2
        \sum_{i \ge 0} 
        \frac{|\mc{G}_{2,2}(\Q_2)_i|}{|\mc{G}_{2,2}(\Q_2)_0|}
        \left(
            1 - \frac{1}{|\varphi(\mc{G}_{2,2}(\Q_2)_i)|}
        \right)
        \\ &= 
        2\left(
            2\cdot \frac{64}{64} \left(1 - \frac{1}{|G_1|}\right) +
            4\cdot \frac{16}{64} \left(1 - \frac{1}{|G_5|}\right) + 
            8\cdot \frac{4}{64} \left(1 - \frac{1}{|G_{13}|}\right) + 
            16\cdot \frac{2}{64} \left(1 - \frac{1}{|G_{29}|}\right)
        \right)
        \\ &=
        d(g_{-1},g_{5},g_{2}).
    \end{align*}
    This proves \cref{thm:bijection between 2-adic homomorphisms and admissible triples}(ii).
\end{proof}

\bibliographystyle{alpha-fullkey}
\nocite{CremonaJonesSutherlandVoight2021LMFDB}
\bibliography{bib.bib}

\end{document}